\documentclass[a4paper,reqno,oneside,10pt]{amsart} 

\usepackage[latin1]{inputenc}

\usepackage{epsfig}
\usepackage{graphicx}
\usepackage{psfrag}

\usepackage{mathrsfs}
\usepackage{amsmath}
\usepackage{amssymb}
\usepackage{extarrows}
\usepackage{enumerate}
\usepackage{mathrsfs}

\usepackage{tikz}
\usepackage{tikz-cd}
\usetikzlibrary{matrix,arrows,decorations.pathmorphing}

\newcommand{\R}{\mathbf{R}}

\newcommand{\Z}{\mathbf{Z}}

\newcommand{\Hyp}{\mathbf{H}}

\def\a{\alpha}
\def\b{\beta}
\def\g{\gamma}
\def\d{\delta}
\def\e{\varepsilon}
\def\l{\lambda}
\def\f{\varphi}

\def\G{\Gamma}
\def\D{\Delta}

\newcommand{\fix}{\mathsf{Fix}}
\newcommand{\Conf}{\mathsf{Conf}}
\newcommand{\Diff}{\mathsf{Diff}}
\newcommand{\B}{\mathsf{B}}

\newcommand{\mradius}{\mathsf{R}}

\newcommand{\hgy}{\mathrm{H}}
\newcommand{\area}{\mathsf{area}}
\newcommand{\sys}{\mathsf{sys}}

\newcommand{\Mod}{\mathsf{Mod}}
\newcommand{\Isom}{\mathsf{Isom}}

\newcommand{\SL}{\mathsf{SL}}

\DeclareMathOperator{\hess}{Hess}

\newcommand{\short}{\mathsf{Sing}}

\newcommand{\arcs}{\mathcal{A}}
\newcommand{\Dcal}{\mathcal{D}}
\newcommand{\Hold}{\mathcal{H}(\l;\R)}
\newcommand{\push}{\mathsf{push}}

\newcommand{\teich}{\mathsf{Teich}}

\newcommand{\poly}{\mathsf{P}}

\newcommand{\Sys}{\mathsf{S}}
\newcommand{\V}{\mathsf{V}}

\newcommand{\rep}{\mathsf{Rep}}
\newcommand{\SO}{\mathsf{SO}}

\newcommand{\mrm}{\mathrm}
\newcommand{\mc}{\mathcal}

\theoremstyle{plain}
\newtheorem{theorem}{Theorem}[section]

\newtheorem{corollary}[theorem]{Corollary}
\newtheorem{proposition}[theorem]{Proposition}

\newtheorem{lemma}[theorem]{Lemma}
\newtheorem*{lemmanonumber}{Lemma}

\theoremstyle{definition}
\newtheorem{definition}{Definition}[section]
\newtheorem{question}{Question}[section]

\newtheorem{remark}{Remark}[section]

\newtheorem*{notation}{Notation}

\title[]{The injectivity radius of hyperbolic surfaces and some Morse functions over moduli spaces}

\begin{document}
\maketitle
 \begin{small}
\begin{flushleft} 
 \textbf{Matthieu Gendulphe}\\
 Dipartimento di Matematica Guido Castelnuovo\\
 Sapienza Universit\`a di Roma\\
 email : matthieu@gendulphe.com
\end{flushleft}\end{small}
\medskip

\begin{abstract}
This article is devoted to the variational study of two functions defined over some Teichm\"uller spaces of hyperbolic surfaces. One is the systole of geodesic loops based at some fixed point, and the other one is the systole of arcs.\par
 For each of them we determine all the critical points. It appears that the systole of arcs is a topological Morse function, whereas the systole of geodesic loops have some degenerate critical points. However, these degenerate critical points are in some sense the obvious one, and they do not interfere in the variational study of the function.\par  

 At a nondegenerate critical point, the systolic curves (arcs or loops depending on the function involved) decompose the surface into regular polygons. This enables a complete classification of these points, and some explicit computations. In particular we determine the global maxima of these functions. This generalizes optimal inequalities due to Bavard (\cite{bavard-toulouse,bavard-manuscripta}) and Deblois (\cite{deblois}). We also observe that there is only one local maximum, this was already proved in some cases by Deblois (\cite{deblois-preprint}).\par
 
 Our approach is based on the \emph{geometric Vorono\"i theory} developed by Bavard (\cite{bavard-crelle}). To use this variational framework, one has to show that the length functions (of arcs or loops) have positive definite Hessians with respect to some system of coordinates for the Teichm\"uller space. Following our previous work \cite{gendulphe-shearing}, we choose Bonahon's shearing coordinates, and we compute explicitly the Hessian of the length functions of geodesic loops. Then we use a characterization of the nondegenerate critical points due to Akrout (\cite{akrout}). \end{abstract}

\renewcommand{\abstractname}{R\'esum\'e}

\section{Introduction}

 Let $(S,m)$ be a hyperbolic surface of finite area without boundary.  The \emph{systole} of $m$ is the length of its shortest closed geodesic, we denote it by $\sys(m)$. It defines a continuous function over the Teichm\"uller space $\teich(S)$, namely the space of isotopy classes of hyperbolic metrics on $S$. It is in general very difficult to study a metric invariant as a function over $\teich(S)$. However, a variational approach of the systole has been initiated by Schmutz (\cite{schmutz}). It has been showed that the systole is a topological Morse function (Akrout \cite{akrout}), its critical points have been characterized, and among them the local maxima (Schmutz \cite{schmutz}, Bavard \cite{bavard-crelle}). This certainly makes the systole a very interesting function.\par
 
 Another interesting metric invariant is the \emph{maximal injectivity radius}, that is the radius of the largest hyperbolic disk isometrically embedded in $(S,m)$, we denote it by $\mradius(m)$. As for the systole, the maximal injectivity radius defines a continuous function over $\teich(S)$, and we know quite a lot about its extremal values: Yamada (\cite{yamada}) determined its infimum when $S$ is orientable, Bavard (\cite{bavard-toulouse}) determined the points realizing its global maximum when $S$ is closed, Bacher and Vdovina (see \cite{vdovina}) computed the number of points realizing the global maximum when $S$ is closed orientable, Girondo and Gonz\'ales-Diez and Nakamura (\cite{girondo,nakamura}) studied the number of extremal disks in a surface realizing the global maximum. More recently, Deblois (\cite{deblois}) determined the global maximum of $\mradius$ in the case of orientable surfaces with cusps, and showed (\cite{deblois-preprint}) that local maxima are global ones. Despite these results we do not have any variational framework for the study of $\mradius$.\par
 
 The systole and the maximal injectivity radius are the extrema of a third metric invariant which is the pointwise version of the systole. Let us fix a point $p$ in $S$. The \emph{systole at $p$} of $m$ is the length of the shortest $m$-geodesic loop based at $p$, we denote it by $\sys_p(m)$. It defines a function over the Teichm\"uller space $\teich(S,p)$ of hyperbolic metrics on $S$ up to isotopy relative to $p$. The systole $\sys(m)$ and the maximal injectivity radius $\mradius(m)$ are related to the extremal values of $\sys_p$ on the fiber $\pi_p^{-1}(m)$ of the projection $\pi_p:\teich(S,p)\rightarrow \teich(S)$:
$$\sys(m)=\inf_{\pi_p^{-1}(m)} \sys_p= \inf\{ \sys_p(m')~;~m'\textnormal{ is isotopic to } m\},$$
$$2\mradius(m)=\sup_{\pi_p^{-1}(m)} \sys_p= \sup\{\sys_p(m')~;~m'\textnormal{ is isotopic to } m\}.$$
The second equality is only true when $S$ is closed.
The first equality uses the assumption that $S$ has empty boundary, otherwise one should take care of the distance $d_m(p,\partial S)$ to the boundary when looking at $\mradius$. We stress the obvious equality $\sup_{\teich(S,p)} \sys_p=\sup_{\teich(S)} \mradius$.\par

 The aim of this article is to provide a variational point of view on $\sys_p$ as a function over $\teich(S,p)$. This function is more complicated than the systole. For instance $\sys_p$ admits some degenerate critical points. Let us denote by $\short(S,p)$ the set of points $[m]_p$ in $\teich(S,p)$ such that among the $m$-systolic loops at $p$ (\emph{i.e.} geodesic loops based at $p$ of length $\sys_p(m)$) there is at least one closed geodesic.

\begin{theorem}
The set of degenerate critical points of $\sys_p$ is $\short(S,p)$. In particular the restriction of $\sys_p$ to the dense open subset $\teich(S,p)-\short(S,p)$ is a topological Morse function. Moreover, a point $[m]_p\in \teich(S,p)-\short(S,p)$ is critical of index $k$ if and only if it has exactly $k$ systolic loops based at $p$ that decompose $(S,m)$ into regular polygons.
\end{theorem}
 
 Actually we prove this theorem for any surface $S$ of finite type with negative Euler characteristic. When $S$ has nonempty boundary, the definition of the Teichm\"uller space is the same with the additional condition that the length of each boundary component is fixed.\par
 We say that a point is \emph{extreme} for $\sys_p$ if it realizes a local maximum of $\sys_p$. The extreme points are exactly the critical points of maximal index, the theorem above implies directly:

\begin{theorem}
Let $S$ be a compact surface with $k\geq 0$ boundary components $b_1,\ldots, b_k$. A point $[m]\in\teich(S,p)$ is extreme for $\sys_p$ if and only if its systolic loops at $p$ divide $S$ into equilateral triangles and one holed monogons. In that case $\sys_p(m)$ is the unique positive solution of the equation
$$6(-2\chi(S)+2-k) \mathrm{arcsin}\left(\frac{1}{2\cosh(x/2)} \right)+ 2\sum_{i=1}^k \mrm{arcsin}\left(\frac{\cosh(\ell(b_i)/2)}{\cosh(x/2)} \right)~ =~2\pi. $$
We allow the length of a boundary component to be zero (the boundary component has to be replaced by a cusp). 
\end{theorem}
 
 This theorem generalizes the results of Bavard and Deblois we mentioned above. Indeed we allow the surface $S$ to have nonempty boundary and to be nonorientable. The works of Bavard and Deblois are based on some decompositions of a hyperbolic surfaces with respect to a point (Vorono\"i decomposition for Bavard, centered dual decomposition for Deblois).\par

 Our approach is radically different. We look at $\sys_p$ as the infimum of the length functions of geodesic loops based at $p$:
$$\sys_p =\inf_{\g\in\pi_1(S,p)} \ell_\g .$$ 
This makes $\sys_p$ a \emph{generalized systole}, that is a function which is locally the minimum of a finite number of \emph{differentiable} functions. Although $\sys_p$ is only continuous, we can use the differentiability of the $\ell_\g$'s to characterize its critical points.\par

 The notion of generalized systole was introduced by Bavard (\cite{bavard-crelle}) who developed a rich variational theory (called \emph{geometric Vorono\"i theory}) for their study ((\cite{bavard-crelle,bavard-smf}). The most important theorems of this theory require some additional \emph{convexity} assumptions on the $\ell_\g$'s. For instance, if the Hessians of the $\ell_\g$'s are positive definite on some open subset $U\subset \teich(S,p)$, then a theorem of Akrout (\cite{akrout}) asserts that $\sys_p$ is a topological Morse function on $U$, and a theorem of Bavard characterizes the extreme points of $\sys_p$ in $U$.\par

 The main difficulty in our approach is to establish such a positivity property on the Hessians of the $\ell_\g$'s.
As it is interesting in itself, we compute explicitly the Hessian of $\ell_\g$ in the shearing coordinates, we show that is positive-semidefinite, and we determine its isotropic cone. This is a generalization of our previous work \cite{gendulphe-shearing}, where we computed the Hessian in the shearing coordinates of the length functions of geodesics over $\teich(S)$. Note that dealing with \emph{pointed} hyperbolic surfaces complicate the situation, because there are no simple system of coordinates on $\teich(S,p)$ like Fenchel-Nielsen coordinates. That's why we use the shearing coordinates, let us explain this point quickly. A maximal geodesic lamination $\l$ divide a hyperbolic surface $(S,m)$ into ideal triangles, and we can reach any other metric $m'$ on $S$ (up to isotopy) by shifting the triangles with respect to each other. The amount of shifting is the \emph{shearing coordinate} of $[m']\in \teich(S)$. As there is a unique ideal triangle up to isometry, we can fix a point on $(S-\l,m)$ and follow it  when shifting the triangles. This gives a system of coordinates on the open subset of $\teich(S,p)$ that corresponds to points $[m]_p$ such that $p$ is not on $\l$. Another advantage of the shearing coordinates is that any linear condition on the lengths of the boundary components define an affine subspace of $\teich(S)$, thus the positivity property of the Hessians of length functions remains true on this affine subspace (the result of Wolpert \cite{wolpert} does not apply for surfaces with fixed boundary lengths).\par

 Once this positivity property is established, we easily find the extreme points of $\sys_p$. Indeed, according to a theorem of Bavard (Theorem~\ref{thm:voronoi}) an extreme point $[m]_p$ has at least $\dim \teich(S,p)+1$ systolic loops at $p$. But, a simple computation of Euler characteristic shows that this is the maximal number of disjoint non homotopic loops based at $p$. We conclude that the systolic loops at $p$ decompose $(S,m)$ as in the theorem above, and one easily obtains the equation giving the global maximum of $\sys_p$. Note that in this approach we do not need to distinguish cases (closed or not, compact or not, with or without boundary, orientable or not).\par
 
 We also determine the other critical points, which have been characterized by Akrout: a point $[m]_p$ is critical if the family of differentials $\mrm d\ell_\g(m)$ of the length functions of its systolic loops satisfy a particular configuration (said \emph{eutactic}) in the cotangent space of . To work with the differentials of the length functions of geodesic loops, we consider a set of homotopy classes $\pi_1(S,p)$ that decomposes $S$ into polygons, then we embed $\teich(S,p)$ into the product of the corresponding spaces of polygons.  The trigonometric relations between the sides and angles give relations between the differentials of the length functions. Without entering into details, let us say that the choice of a set of homotopy classes corresponds to the choice of a \emph{minimal class} (a useful notion when studying generalized systoles).\par
 
 The same approach applies to another metric invariant: the \emph{systole of arcs}. Let us introduce some notations. We assume that $S$ has $k\geq 1$ boundary components, and we fix a nonempty set $B=\{b_1,\ldots, b_l\}$ of boundary components. We denote by $\teich(S,B)$ the Teichm\"uller space of hyperbolic metrics such that the sum $\ell(B)=\ell(b_1)+\ldots+\ell(b_l)$ is fixed, as are the lengths $\ell(b_{l+1}),\ldots,\ell(b_k)$. We denote by $\arcs(B)$ the set of isotopy classes of essential arcs whose endpoints belong to some boundary components in $B$. We denote by $\sys_B$ the length of the shortest essential arc whose isotopy class belongs to $\arcs(B)$.
 
\begin{theorem}
The function $\sys_B$ is a topological Morse function over $\teich(S,B)$.
 The critical points for $\sys_B$ are in bijection with the systems of arcs in $\arcs(B)$ that decompose $S$ into polygons. More precisely, to any such arc system $\{\a_1,\ldots,\a_n\}$ corresponds the unique point $[m]\in\teich(S,B)$ such that the arcs realizing $\sys_B$ belong to $\a_1,\ldots,\a_n$ and decompose $(S,m)$ into semi-regular right-angled polygons (a polygon is semi-regular if its lengths alternate between two values).
\end{theorem} 
 
As a direct consequence we get  

\begin{theorem}
A point $[m]\in\teich(S,B)$ is extreme if and only if the systolic arcs of $m$ divide $S$ into right-angled hexagons and right-angled bigons containing a boundary component which is not in $B$. Moreover $\sys_B(m)$ is the unique solution of the equation
$$6[-2\chi(S)-(k-l)] \mrm{arcsinh}\left(\frac{1}{2\sinh x}\right) + 2\sum_{i=l+1,\ldots,k} \mrm{arcsinh}\left(\frac{\cosh(b_i/2)}{\sinh(x/2)}\right)=\ell(B).$$ 
In particular, extreme points of $\sys_B$ are in bijection with maximal system of arcs in $\arcs(B)$.
\end{theorem} 
 
 This generalizes an optimal inequality of Bavard (\cite{bavard-manuscripta}) that corresponds to the case in which $B$ is the full set of boundary components. The inequality of Bavard is not very accurate when some boundary components are very short, our theorem works well in that case.\par
 
 As we know all the critical points of the systole of arcs, we get some information on the topology of the moduli  space. Taking $B$ as the full set of boundary components, it comes that $\teich(S,B)$ deforms into a cell complex of dimension the virtual cohomological dimension of the modular group $\Mod(S)$. The virtual cohomological dimension of $\Mod(S)$ has been computed by Harer (\cite{harer}), who already constructed a $\Mod(S)$-invariant deformation retract which is combinatorially equivalent to our cell complex. However, we believe that it is interesting to have another perspective.\par
 
  The proofs of the above theorems follow exactly the same line as for $\sys_p$, but they are more simple as they do not involve Teichm\"uller spaces of pointed surfaces. Therefore we have chosen to treat first the case of $\sys_B$.\par

The paper is organized as follow: in part~1 we set up the notations and definitions, we recall the most important features of Bavard's \emph{geometric Vorono\"i theory}, and we apply it to find the extreme points of $\sys_B$. In part~2, we determine the other critical points of $\sys_B$, this is done through embeddings of $\teich(S,B)$ into products of spaces of right-angled polygons. The part~3 is devoted to the study of $\sys_p$, in particular we compute the Hessian of the length functions of geodesic loops in the shearing coordinates. 

\tableofcontents

\part{Teichm\"uller spaces and generalized systoles}         

\section{Teichm\"uller spaces}

\subsection{Definitions}
 Let $S$ be a surface of finite type with negative Euler characteristic. A \emph{hyperbolic metric} on $S$ is a complete metric of constant curvature $-1$. In this article we always assume that the boundary $\partial S$ is geodesic and the area $\area(m)$ finite. We denote by $b_1,\ldots, b_k$ ($k\geq 0$) the boundary components of $S$.\par
 
  The \emph{Teichm\"uller space} of $S$ is the space of hyperbolic metrics on $S$ up to isotopy, we denote it by $\teich(S)$. It is diffeomorphic to a ball of dimension $-3\chi(S)$. The \emph{modular group} $\Mod(S)$ is the group of isotopy classes of diffeomorphisms that preserve each boundary component. When $S$ is orientable we only consider orientation preserving diffeomorphisms. The modular group acts properly and discontinuously on $\teich(S)$, and the quotient is the moduli space of hyperbolic metrics on $S$.\par
  
 When $S$ has nonempty boundary, we can define some other Teichm\"uller spaces by adding conditions on the the lengths of the boundary components. Our results deal with the following kind of Teichm\"uller spaces: let $B=\{b_1,\ldots,b_l\}$ be a set of boundary components of $S$, we denote by $\teich(S,B)$ the space of isotopy classes of hyperbolic metrics on $S$ such that the sum of the lengths of the boundary components in $B$ is fixed, as is the length of any other component. The Teichm\"uller space $\teich(S,B)$ is a submanifold of codimension $(k+1)$ of $\teich(S)$. When $B$ is the full set of boundary components of $S$ then we use the notation $\teich(S,\partial S)$.\par
  
  Given $\g$ a non trivial isotopy class of closed curves on $S$, we denote by $\ell_\g:\teich(S)\rightarrow \R_+^\ast$ the smooth function that associates to a point $[m]\in\teich(S)$ the length $\ell_\g(m)$ of the unique $m$-geodesic in $\g$. Similarly, an isotopy class $\a$ of essential arcs on $S$ defines a smooth function $\ell_\a:\teich(S)\rightarrow \R_+^\ast$ that associates to $[m]\in\teich(S)$ the length of the unique $m$-geodesic arc in $\a$ orthogonal to $\partial S$.\par
   We recall that an \emph{arc} is a smooth embedding $(I,\partial I)\rightarrow (S,\partial S)$ transverse to $\partial S$, and an arc is \emph{essential} if it can not be deformed into $\partial S$. In the sequel we only consider isotopy classes of essential arcs.\par
    
\subsection{Shearing coordinates and Hessian of length functions}
Let $\l$ be a maximal geodesic lamination on $S$, it necessarily contains the boundary components of $S$. For any hyperbolic metric $m$ on $S$ the complement $\tilde S-\tilde \l$ is a disjoint union of ideal triangles. The way the triangles are glued together is encoded by a transverse H\"older distribution for $\l$, and this defines a smooth embedding of $\teich(S)$ into the linear space $\Hold$ of transverse H\"older distribution for $\l$. The image of this embedding is an open convex cone. We call shearing coordinates for $\teich(S)$ any linear system of coordinates for $\Hold$.\par
 
 In this article we will not manipulate directly transverse H\"older distributions. So we refer to the original works \cite{bonahon-topology,bonahon-toulouse} of Bonahon for a complete treatment of the subject. For our purpose, the short exposition we made in \cite[\textsection 2]{gendulphe-shearing} would be largely sufficient.\par
  We need the shearing coordinates in order to establish convexity properties of the length functions of geodesics or arcs. As explained in the introduction these properties are crucial in our study. In \cite{gendulphe-shearing} we proved the following theorem:
 
\begin{theorem}\label{thm:shearing}
Let $\g$ be an isotopy class of closed curves that intersect every leaf of $\l$. The Hessian of $\ell_\g$ in the shearing coordinates is positive-definite over $\teich(S)$.
\end{theorem} 

 To simplify the presentation we proved this theorem only in the case of closed orientable surfaces, but the proofs work for all the cases we considered. Alternatively one can consider the double of a surface which has nonempty boundary of is orientable.

The length functions of the boundary components are linear functions in the shearing coordinates (see \cite{bonahon-toulouse}). Thus, in the shearing coordinates, the Teichm\"uller $\teich(S,B)$ is the intersection of $\teich(S)$ with a linear subspace. We deduce that the theorem above is still true for this kind of Teichm\"uller spaces.

\section{Geometric Vorono\"i theory}

In this section we recall the main features of the \emph{geometric Vorono\"i theory} developed by Bavard in \cite{bavard-crelle,bavard-smf}.
The aim of this theory is to study the variational properties of the \emph{generalized systoles} defined below.\par

Without entering into details, let us say that the classical example of a generalized systole is the \emph{Hermite invariant} that defines a $\SL(n,\Z)$-invariant function over the symmetric space $\SL(n,\R)/\SO(n,\R)$. We refer to the work of Bavard \cite{bavard-crelle} for a detailed description of the analogy between the Hermite invariant as a function over $\SL(n,\R)/\SO(n,\R)$ and the systole as a function over $\teich(S)$.

 \subsection{Generalized systoles}
Let $V$ be a manifold, $C$ be a set, and $\{f_s:V\rightarrow \R~;~s\in C\}$ be a family of functions of class $\mc C^1$ indexed by $C$. We say that $(f_s)_{s\in C}$ is a \emph{system of length functions} if it satisfies the following condition of local finiteness: \emph{for any $p\in V$ and any $L>0$ there exists a neighborhood $U$ of $p$ such that the set $\{s\in C~;~f_s(p)\leq L\}$ is finite}. Given a system of length functions $(f_s)_{s\in C}$ we are interesting in its infimum 
\begin{eqnarray*}
\mu & = & \inf_{s\in C} f_s.
\end{eqnarray*}
Such a function is called a \emph{generalized systole}. The local finiteness condition implies that $\mu$ is continuous, however it is not differentiable in general.\par

 Let us give some examples of generalized systoles over Teichm\"uller spaces. The first one is of course the systole itself which defines a function $\sys:\teich(S)\rightarrow\R_+^\ast$ that can be written as the minimum of the length functions of geodesics:
\begin{eqnarray*}
\sys & = & \min_{\g} \ell_\g,
\end{eqnarray*}
where $\g$ runs over the set of non trivial isotopy classes of closed curves on $S$. This function is clearly $\Mod(S)$-invariant. When $S$ has empty boundary, the induced function on the moduli space is proper (Mumford compactness theorem), in particular admits a global maximum. \par

 Now let us assume that $S$ has nonempty boundary. The systole of arcs is the minimum of the length functions of the isotopy classes of essential arcs:
\begin{eqnarray*}
\sys_\partial & = & \min_{\a} \ell_\a,
\end{eqnarray*}  
where $\a$ runs over the set of essential isotopy classes of arcs on $S$. The restriction of $\sys_\partial$ to the Teichm\"uller space $\teich(S,\partial S)$ of hyperbolic metrics on $S$ with fixed total boundary length is bounded, and actually admits a global maximum (this is not difficult to see). The systole of arcs is $\Mod(S)$-invariant but not proper (except in the case of the one-holed torus).\par

 We end this paragraph with a function which is not a generalized systole. Let us assume that $S$ is closed. A \emph{pants decomposition} of $S$ is a maximal family $\g=\{\g_1,\ldots,\g_k\}$ of disjoint isotopy classes of essential simple closed curves. We define the length function of $\g$ as the length of its longest component: 
 \begin{eqnarray*}
 \ell_\g & = & \max_{i=1,\ldots,k} \ell_{\g_i}.
 \end{eqnarray*} 
The family of length functions of pants decomposition of $S$ satisfies the local finiteness condition, but its infimum is not a generalized systole. Indeed, the length function of a pants decomposition is not everywhere differentiable. Still, the infimum of length functions of pants decomposition shares many properties with generalized systoles (see \cite{schmutz, gendulphe-annalen}), its global maximum is called the Bers' constant of $S$.

\subsection{Caracterisation of extreme points}
 We consider a generalized systole $\mu$ associated to a system of length functions $\{f_s:V\rightarrow \R~;~s\in C\}$. For any point $p\in V$ we set 
$$\Sys(p)~=~\{s\in C~;~\mu(p)=f_s(p)\}.$$
By local finiteness, any point $p\in V$ has a neighborhood $U$ on which
$$\mu ~= \min_{s\in \Sys(p)} f_s.$$
The length functions $(f_s)_{s\in \Sys(p)}$ are differentiable, so we expect the behaviour of $\mu$ in a small neighborhood of $p$ to depend on the configuration of the differentials $(\mrm df_s(p))_{s\in \Sys(p)}$ in the cotangent space $T_p^\ast V$. This is exactly the case.\par
 We first introduce the following vectorial properties:

\begin{definition}
Let $E$ be a finite dimensional real vector space, and $\mc F$ be a finite family of vectors of $E$. We denote by $K$ the convex hull of $\mc F$. We say that
\begin{itemize}
\item $\mc F$ is \emph{perfect} if it affinely spans $E$,
\item $\mc F$ is \emph{eutactic} if the affine interior of $K$ contains the origin of $E$. 
\end{itemize}
\end{definition}

Note that these properties are invariant under linear isomorphism. To these properties correspond two kinds of points that are interesting with respect to $\mu$:  

\begin{definition}
We say that a point $p\in V$ is \emph{perfect} (resp. \emph{eutactic}) if the family of differentials $(\mrm df_s(p))_{s\in \Sys(p)}$ is perfect (resp. eutactic) in the cotangent space $T_p^\ast V$
\end{definition}

 We will shortly see that perfection and eutaxy characterize the following points:

\begin{definition}
We say that a point $p\in V$ is \emph{extreme} (resp. \emph{strictly extreme}) if it realizes a local maximum of $\mu$ (resp. a strict local maximum).
\end{definition}

 We do not have enough assumptions on the $(f_s)_{s\in C}$ to obtain interesting results.
So we introduce the following \emph{condition $(H)$}: \emph{any point $p\in V$ has a neighborhood $U$ equipped with a Riemannian metric such that the Hessians of $(f_s)_{s\in \Sys(p)}$ are positive-definite on $U$.} We now state the following general theorem due to Bavard (\cite[Proposition~2.3]{bavard-crelle}): 
 
\begin{theorem}[Bavard]\label{thm:voronoi}
If the condition (H) is satisfied, then a point is extreme if and only if it is perfect and eutactic. Moreover every extreme point is strictly extreme.
\end{theorem}

\begin{remark}
\begin{enumerate}
\item Many particular cases of this theorem were already known. Let us mention that Schmutz (\cite{schmutz}) established a similar theorem for the systole.
\item The theorem is true under the weaker assumption that the $(f_s)_{s\in S_p}$ are stricly convex along the geodesics of $U$ (see \cite{bavard-crelle,bavard-smf}).
\end{enumerate}
\end{remark}

 From the definition of perfection we get the following corollary which is very useful to find extreme points:

\begin{corollary}\label{cor:voronoi}
If $p\in V$ is extreme then the cardinal of $\Sys(p)$ is at least $\dim V+1$.
\end{corollary}

 The Theorem~\ref{thm:shearing} shows that the condition $(H)$ is satisfied in the cases of the systole and of the systole of arcs.

\section{Extreme points of the systole of arcs}

 Let us assume that $S$ has nonempty boundary, and let $B=\{b_1,\ldots, b_l\}$ be a nonempty set of boundary components of $S$. We recall that $\teich(S,B)$ is the space of isotopy classes of hyperbolic metrics on $S$ such that the sum $\ell(B)$ of the lengths of the boundary components in $B$ is fixed, and such that the lengths $\ell(b_{l+1}),\ldots,\ell(b_k)$ of the other boundary components are fixed. We allow the length of a boundary component to be zero, which means that the corresponding boundary component should be replaced by a cusp.\par
  We denote by $\arcs(B)$ the set of isotopy classes of essential arcs in $S$ whose endpoints belong to some boundary components in $B$. We now introduce the following generalized systole:
$$\sys_B ~=~ \inf_{\a\in \arcs(B)} \ell_\a.$$
 It is not difficult to show that $\sys_B$ admits a global maximum over $\teich(S,B)$. Note that $\sys_B=\sys_{\partial S}$ when $B$ is the full set of boundary components of $S$.\par

 To illustrate the theorem of the previous section, we determine all the extreme points of $\sys_B$ over the Teichm\"uller space $\teich(S,B)$:\par
 
\begin{theorem}\label{thm:max-arcs}
A point $[m]\in\teich(S,B)$ is extreme if and only if the systolic arcs of $m$ divide $S$ into right-angled hexagons and right-angled bigons with a boundary component which is not in $B$. Moreover $\sys_B(m)$ is the unique solution of the equation
$$6[-2\chi(S)-(k-l)] \mrm{arcsinh}\left(\frac{1}{2\sinh x}\right) + 2\sum_{i=l+1,\ldots,k} \mrm{arcsinh}\left(\frac{\cosh(b_i/2)}{\sinh(x/2)}\right)=\ell(B).$$ 
In particular, extreme points of $\sys_B$ are in bijection with maximal arcs system.
\end{theorem} 
  
\begin{remark}
\begin{enumerate}
\item We observe that there is only one local maximum. 
\item We actually prove that the extreme points are the only perfect points.
\end{enumerate}
\end{remark}  
  
  As a corollary we obtain the following optimal inequality due to Bavard (\cite{bavard-manuscripta}):
\begin{corollary}[Bavard]
Let $m$ be a hyperbolic metric on $S$, then we have
\begin{eqnarray*}
2 \sinh(\sys_a(m)/2)\sinh(-\ell(\partial S)/12\chi(S)) & \leq & 1,
\end{eqnarray*}
with equality if and only the arcs realizing $\sys_\partial$ divide the surface into right-angled hexagons.
\end{corollary}

We start with an obvious remark:

\begin{lemma}
Let $m$ be any hyperbolic metric on $S$. Let $\a$ and $\b$ be two geodesic arcs that are of minimal length in their homotopy class (or equivalently orthogonal to $\partial S$). Let us assume that they intersect, and let $\g$ any piecewise geodesic arc made with exactly one segment of $\a$ and one segment of $\b$. Then $\g$ is essential, that is it can not be deformed into $\partial S$.
\end{lemma} 

From this lemma it comes that the arcs realizing $\sys_B $ are pairwise disjoint. Then it is easy to compute their maximal number:

\begin{lemma}
Let $m$ be a hyperbolic metric on $S$.
The maximal number of disjoint geodesic arcs orthogonal to $\partial S$ whose isotopy class belong to $\arcs(B)$ is equal to $-3\chi(S)-(k-l)$, which is the dimension of $\teich(S,B)$ plus $1$. Such a family of arcs divide $S$ into right-angled hexagons and bigons, each bigon containing a boundary component that is not in $B$. The number of hexagons is $-2\chi(S)-(k-l)$ .
\end{lemma}

\begin{proof}
The only non trivial statement is on the cardinal of the maximal family of arcs.
To simplify we fix the lengths of all the boundary components equal to zero. Any isotopy class of essential arc is realized by a unique geodesic arc going from one cusp to another, and these arcs realize the intersection number of their isotopy classes.\par
 A maximal family of disjoint geodesic arcs going from one cusp to another divide the surface into ideal triangles and ideal monogons containing one cusp. An ideal mongon is obtained from an ideal triangle by identifying two sides. Note that there are $k-l$ monogons. Ideal triangle and monogons are of area $\pi$. As the area of the surface is $-2\pi\chi(S)$ we find
\begin{eqnarray*}
-2\chi(S) & = & t +(k-l), \\
 2e                & = &  3t+(k-l)  
\end{eqnarray*} 
where $t$ is the number of ideal triangles, and $e$ the number of arcs. It is then very easy to conclude.
\end{proof} 

The proof of the theorem is an application of the previous lemmas and of Theorem~\ref{thm:voronoi}. 

\begin{proof}[Proof of the Theorem~\ref{thm:max-arcs}]
Let $[m]\in\teich(S,B)$ be an extreme point of $\sys_B$. According to Corollary~\ref{cor:voronoi} the hyperbolic surface $(S,m)$ has at least $\dim\teich(S,B)+1$ arcs realizing $\sys_B(m)$. From the previous lemmas, we deduce that these arcs divide $S$ into right-angled hexagons and right-angled bigons, each bigon containing one boundary component that does not belong to $B$.\par
 Each hexagon has three sides of length $\sys_B(m)$, and three other sides of length $\b$ given by $2\sinh(\b/2)\sinh(\sys_B(m))=1$. We recall that there are $-2\chi(S)-(k-l)$ hexagons. The bigon containing the boundary component $b_i$ has a side of length $\sys_B(m)$ and the other of length $\b_i$ given by $\sinh(\b_i/2)\sinh(\sys_B(m)/2)=\cosh(\ell(b_i)/2)$. One easily concludes.
\end{proof}

 Let us note that $\sys_B$ and $\teich(S,B)$ have been chosen so that the maximal cardinality of a set of arcs realizing $\sys_B$ is exactly $\teich(S,B)+1$. This is a key point in our proof.

\section{Topological singularities of generalized systoles}

\subsection{Topological Morse functions}
We recall the notion of topological Morse function introduced by Morse himself under the name \emph{topologically nondegenerate functions}. We refer to the work of Morse (\cite{morse-fund,morse-janalyse}) for proofs.\par

Let $F:M\rightarrow \R$ be a continuous function over a topological manifold $M$. A point $p\in M$ is \emph{regular} if there is a topological chart $\f:(U,p)\subset M\rightarrow (\R^n,0)$ such that $F\circ \f-F(p)$ is the restriction of a nonzero linear form. A \emph{critical point} is a point which is not regular. A critical point $p$ is \emph{nondegenerate} if there is a topological chart $\f:(U,p)\subset M\rightarrow (\R^n,0)$ such that $F\circ \f-F(p)$ is the restriction of a nondegenerate quadratic form. The number of negative eigenvalues of the quadratic form is called the \emph{index} of $F$ at $p$. A \emph{topological Morse function} $F:V\rightarrow \R$ is a continuous function whose critical points are all nondegenerate.\par
 
 Note that there are topological Morse functions on $\R$ that are not differentiable, so this notion is more general than the usual one. The classical theorems of Morse theory remains true for topological Morse function. Let us now assume that $F:M\rightarrow \R_+$ is proper, namely its supremum is $+\infty$ and is reached when going at infinity in $M$ (\emph{i.e.} leaving all compact subsets). Then Morse showed that $M$ can be deformed into a compact cell complex $K\subset M$, in particular the inclusion $K\hookrightarrow M$ induced an isomorphism on homology $\hgy_\ast(K;\Z)\simeq \hgy_\ast(M;\Z)$. The topology of $K$ is completely determined by the critical points of $F$, and the homology of $K$ is isomorphic to the homology of the complex of critical points. Note that the dimension of $K$ is equal to the maximal index of a critical point. 

\subsection{Akrout's theorem}
 Let us now come back to generalized systoles. Following our previous notations, we consider the generalized systole $\mu$ associated to a system of length functions $\{f_s:V\rightarrow \R~;~s\in C\}$. Under the condition $(H)$, the topological singularities of the generalized systole are completely described by Akrout's theorem (\cite{akrout}):  
   
\begin{theorem}[Akrout]\label{thm:akrout}
If (H) is satisfied, then $\mu$ is a topological Morse function over $V$. Moreover its critical points of index $r$ are exactly the eutactic points of rank $r$ (the rank of an eutactic point $p$ being the vectorial rank of the differentials $(\mrm df_s(p))_{s\in \Sys(p)}$).
\end{theorem}

\begin{remark}
The theorem implies that the eutactic points are isolated. In many cases it is not difficult to deduce the finiteness, eventually up to the action of some group like $\Mod(S)$.
\end{remark}

From his remarkable theorem, and from the result of Wolpert (\cite{wolpert}), Akrout deduced that the systole $\sys:\teich(S)\rightarrow\R$ is a topological Morse function. This implies that the existence of a $\Mod(S)$-invariant deformation of $\teich(S)$ into a cell complex $K\subset \teich(S)$. The codimension of $K$ is equal to the minimal rank of an eutactic point. When $S$ is a closed oriented surface of genus $g\geq 2$, the virtual codimension of $\Mod(S)$ is $4g-5$. So, the following classical question is very natural:

\begin{question}
What is the minimal rank of a critical point of $\sys$?
\end{question}

 As we already pointed out, the theorem of Wolpert does not work when we impose some conditions on the lengths of the boundary components. But our Theorem~\ref{thm:shearing} applies in that case, and shows that the systole is still a topological Morse function.
Our theorem also implies that 

\begin{theorem}
The function $\sys_B$ is a topological Morse function over $\teich(S,B)$.
\end{theorem}

\subsection{The search for critical points}

 In order to obtain any homological information, one has still to determine the critical points.\par
  According to a general principle, points with many symmetries should interesting. In our context this is illustrated by the following lemma (\cite[Corollaire~1.3]{bavard-smf}):

\begin{lemma}[Bavard]
Let $G$ be a finite group of diffeomorphisms of $V$ that is acting by precomposition on $(f_s)_{s\in C}$. If $p\in V$ is an isolated fixed point of $G$ then $p$ is eutactic.
\end{lemma}

\begin{proof}[Proof following Bavard]
The group $G$ is finite and fixes $p$, therefore it acts by isometries with respect to some Riemannian metric on some neighborhood of $p$. The family of gradients $(\nabla f_s)_{s\in S_p}$ is stabilized by $G$, because $G$ fixes $p$. So $G$ fixes the barycenter of $(\nabla f_s)_{s\in S_p}$ which is necessarily the origin of $T_pV$. Otherwise the geodesic tangent to this vector would be fixed pointwise by $G$, contradicting the assumption of the lemma. We conclude that the origin of $T_p V$ belongs to the affine interior of the convex hull of $(\nabla f_s)_{s\in S_p}$, so $p$ is eutactic.
\end{proof}

 When studying $\mu$ globally over $V$ it is interesting to decompose it into minimal classes:
\begin{definition}
Let $A$ be a subset of $C$. The \emph{minimal class} of $A$ is the set of points $p\in V$ such that $\Sys(p)=A$.
\end{definition}

The minimal classes define a partition of $V$. It is natural to look for critical inside a fixed minimal class. It appears that a critical point should verify the following necessary condition:

\begin{lemma}[Bavard]\label{lem:eutactic-minimum}
If the condition $(H)$ is satisfied, then any eutactic point realizes a strict minimum of $\mu$ in its minimal class.
\end{lemma}

\begin{lemma}\label{lem:eutactic-unicity}
If any two points in the minimal class of $A$ can be joined by a smooth curve $c:I\rightarrow V$ along which the $(f_s)_{s\in A }$ are strictly convex, then the minimal class of $A$ contains at most one eutactic point.
\end{lemma}

\part{Critical points of the systole of arcs}         

 In this part we determine all the critical points of $\sys_B$:

\begin{theorem}
 The critical points for $\sys_B$ are in bijection with the systems of arcs of $\arcs(B)$ that decompose $S$ into polygons. More precisely, to any such arc system $\{\a_1,\ldots,\a_n\}$ corresponds the unique point $[m]\in\teich(S,B)$ such that the arcs realizing $\sys_B$ belong to $\a_1,\ldots,\a_n$ and decompose $(S,m)$ into semi-regular right-angled polygons (a polygon is semi-regular if its lengths alternate between two values).
\end{theorem} 

\begin{remark}
The unicity part of the statement is a direct application of Lemma~\ref{lem:eutactic-unicity} and Theorem~\ref{thm:shearing}.
\end{remark}

 We recall another time our assumptions and notations. The surface $S$ is compact with negative Euler characteristic. We assume that $S$ has nonempty boundary, and we consider a nonempty set $B=\{b_1,\ldots, b_l\}$ of boundary components of $S$. We recall that $\teich(S,B)$ is the Teichm\"uller space of hyperbolic metric on $S$ that satisfy two conditions: the sum $\ell(B)$ of the lengths of the boundary components in $B$ is fixed, as is the length of any other boundary component. We allow a boundary component to have length zero. An \emph{arc system} is a family of disjoint isotopy classes of essential arcs.\par

 According to Akrout theorem (Theorem~\ref{thm:akrout}) the critical points of the $\sys_B$ are exactly its eutactic points. 

\begin{lemma}
Let $[m]\in\teich(S,B)$ be an eutactic point of $\sys_B$, then the arcs realizing $\sys_B$ decompose $S$ into polygons.
\end{lemma} 

\begin{proof}
It the arcs realizing $\sys_B(m)$ do not decompose $S$ into polygons, then they do not intersect a simple closed geodesic $\g$. A small twist along $\g$ does no change the value of $\sys_B(m)$ and its minimal class. We conclude with Lemma~\ref{lem:eutactic-minimum}.
\end{proof} 
 
So, inorder to prove the theorem, we have just to show that the points of $\teich(S,B)$ described in the theorem are indeed eutactic.
To simplify the exposition we only consider the case where $B$ is the full set of boundary components. But all our proofs extend readily to other cases.

\section{Spaces of right angled polygons}

 In the sequel \emph{polygon} means a convex polygon of the hyperbolic plane $\Hyp$.

\subsection{Spaces of right angled polygons}
Given a number $n\geq 5$, we denote by $\poly(n)$ the set of isometry classes of \emph{marked} right angled polygons with $n$ sides, where a marking consists in a cyclic numbering of the sides. The cyclic group $\Z/n\Z$ is acting on $\poly(n)$ by permutation of the marking.\par

 Let $\ell_i:\poly(n)\rightarrow\R$ be the length function of the $i$-th side, the map $(\ell_1,\ldots,\ell_n)$ defines an injection from $\poly(n)$ to $\R^n$. The common perpendicular between the first side and the sides numbered from $4$ to $n-2$ decompose a polygon in $\poly(n)$ into marked right angled pentagons.
 Let $h_i$ be the length of the common perpendicular between the first side and the $i$-th side ($i\neq 2,n-2$). As the isometry class of right angled pentagon is determined by the length of any two sides, it comes that the map $(\ell_3,h_4,h_5,\ldots,h_{n-2},\ell_{n-1})$ establishes a bijection between $\poly(n)$ and $(\R_+^\ast)^{n-3}$. The lengths $\ell_i$ are easily expressed in terms of the coordinates $(\ell_3,h_4,h_5,\ldots,h_{n-2},\ell_{n-1})$, therefore the image of $\poly(n)$ by $(\ell_1,\ldots,\ell_n)$ is a smooth submanifold of codimension $3$.\par
 
 We endow $\poly(n)$ with this smooth structure. Actually, we won't distinguish $\poly(n)$ from its image in $\R^n$. The action of $\Z/n\Z$ on $\poly(n)$ is the restriction of the action of $\Z n/\Z$ on $\R^n$ by permutation of coordinates. This action is by isometries with respect to the canonical Euclidean structure of $\R^n$, so $\Z/n \Z$ acts by isometries on $\poly(n)$ with respect to the induced Riemannian metric.
  
\subsection{Tangent space}\label{sec:tangent}
We fix an index $i\in\{1,\ldots,n\}$ and a point $P$ in $\poly(n)$. We denote by $f_i:\poly(n)\rightarrow\R$ the length function of the common perpendicular to the sides $i-1$ and $i+2$. This common perpendicular divide any element of  $\poly(n)$ into two right angled marked polygons, one of them being a pentagon. A variation of the pentagon while keeping the complementary polygon identical defines a germ of curve in $\poly(n)$:
 $$\begin{array}{cccc}
c  : & (-\e,\e) & \longrightarrow & \poly(n) \\
      & t & \longmapsto & (\ell_1(t),\ldots,\ell_n(t))
\end{array}$$
satisfying $c(0)=P$. Keeping the complementary polygon unchanged is equivalent to keeping the functions  $f_i$ and $\ell_j$ with $j\notin\{i-1,i,i+1,i+2\}$ constant along $c$.
The following equalities define a germ of curve:
$$\left\{\begin{array}{rll}
(\ell_1(0),\ldots,\ell_n(0)) & = & (\ell_1(P),\ldots,\ell_n(P)),\\
\ell_j(t) & = & \ell_j(0)\ \textnormal{for}\ j\notin\{i-1,i,i+1,i+2\}, \\
\ell_i(t) & = & \ell_i(0)+t, \\
\sinh\ell_i(t)\sinh\ell_{i+1}(t) & = & \cosh f_i(P). \\
\end{array}\right. $$ 
The lengths $\ell_{i-1}$ and $\ell_{i+2}$ are expressed in terms of $\ell_i,\ell_{i+1},f_i$ using trigonometric identities:  
$$\left\{\begin{array}{rll}
\ell_{i-1}(t) & = & \a+\b(t)\quad \textnormal{with}\ \a>0\ \textnormal{constant and} \\
 \sinh\b(t) \sinh f_i(P) & = & \cosh\ell_{i+1}(t).  \\
 \ell_{i+2}(t) & = & \g+\d(t)\quad \textnormal{with}\ \g>0\ \textnormal{constant and} \\
 \sinh\d(t) \sinh f_i(P) & = & \cosh\ell_{i}(t).  
\end{array}\right.$$ 
Differentiating the above equalities we get:
$$\left\{\begin{array}{rll}
\ell_i'(t) & = & 1,\\
\coth\ell_i(t)\ell_i'(t) & = & -\coth\ell_{i+1}(t)\ell_{i+1}'(t), \\
\ell_{i-1}'(t) &= & \b'(t), \\
\cosh\b(t) \sinh f_i(P) \b'(t) & = & \sinh\ell_{i+1}(t) \ell_{i+1}'(t), \\
\ell_{i+2}'(t) &= & \d'(t), \\
\cosh\d(t) \sinh f_i(P) \d'(t) & = & \sinh\ell_i(t) \ell_i'(t).
\end{array}\right.$$
With the following simplification:
$$\b'(t) = - \frac{\sinh\ell_{i+1}(t)}{\cosh\b(t)\sinh f_i(P)} \frac{\tanh\ell_{i+1}(t)}{\tanh\ell_{i}(t)}= - \frac{1}{\sinh\d(t)\sinh f_i(P)} \frac{\tanh\ell_{i+1}(t)}{\tanh\ell_{i}(t)}=-\frac{\tanh\ell_{i+1}(t)}{\sinh\ell_i(t)},$$
we obtain:
$$c'=\frac{1}{\sinh\ell_i\cosh\ell_{i+1} }\left(0,\dots,0, -\sinh\ell_{i+1},\sinh\ell_i \cosh\ell_{i+1},-\sinh \ell_{i+1}\cosh\ell_i,\sinh\ell_i,0,\ldots,0\right).$$
The tangent vector $c'(0)$ belongs to $T_P\poly(n)$. Thus we define a vector field $u_i$ on $\poly(n)$ by setting $u_i(P)=\sinh\ell_i(P)\cosh\ell_{i+1}(P) c'(0)$ for every $P\in \poly(n)$. Clearly, any choice of $n-3$ elements in $\{u_i~;~ i=1,\ldots,n\}$ gives a trivialisation of the tangent bundle $T\poly(n)$. 

\subsection{Semi-regular polygons}
We say a right angled polygon is \emph{semi-regular} if the lengths of its sides alternate between two values. A right angled polygon $P\in\poly(2n)$ is \emph{semi-regular} if and only if it is fixed by the diffeomorphism $\f:\poly(2n)\rightarrow\poly(2n)$ defined by  
$$\f  (\ell_1,\ell_2,\ldots,,\ell_{2n-2},\ell_{2n-1},\ell_{2n})  = (\ell_3,\ell_4,\ldots,\ell_{2n},\ell_1,\ell_{2}).$$
This diffeomorphism corresponds to the action of the cycle $i\mapsto i+2$ of $\Z/n \Z$. As the fixed point locus of an isometry, the set of semi-regular right angled polygons is a totally geodesic submanifold $\V(2n)$ of $\poly(2n)$ with respect to the metric induced by the canonical Euclidean metric of $\R^{2n}$. A semi-regular polygon is completely determined by the lengths $\ell_1$ and $\ell_2$, that are related by the identity $\sinh(\ell_1/2)\sinh(\ell_2/2)=\cos(\pi/n)$. We deduce that $\V(2n)$ is diffeomorphic to $\R$, so it is a geodesic of $\poly(2n)$.

\subsection{A relation between the differentials of the length functions}\label{sec:relation}
Let us look at the $1$-forms $\sum_{i=0}^n \mrm d \ell_{2i+1}$ and $\sum_{i=0}^n \mrm d \ell_{2i}$ defined over $\poly(2n)$. 
These $1$-forms are invariant under the action of $\f$. At a point of $\V(2n)$, the vectors $u_i$ have a very simple expression, so that we can evaluate the $1$-forms on these vectors. We find that the following relation of proportionality
$$-\frac{1+\cosh\ell_2}{\sinh\ell_2}\left(\sum_{i=0}^n \mrm d \ell_{2i} \right) =\frac{1+\cosh\ell_1}{\sinh\ell_1} \left(\sum_{i=0}^n \mrm d \ell_{2i+1} \right)$$
is valid over $\V(2n)$. Let us stress that the $1$-forms on the two sides of the equality agree on the tangent space of $\poly(2n)$ at any point of $\V(2n)$, this is much stronger than an equality on the tangent space of $\V(2n)$.

\section{Construction of the critical points}

In this section we show that, if the arcs realizing $\sys_{\partial S}(m)$ decompose $S$ into semi-regular right-angled polygons, then the point $[m]\in\teich(S,\partial S)$ is eutactic.

\subsection{Embedding of the Teichmüller space}
 Let $\a\in\arcs_0(\Sigma)$ be a system of arcs that define a cell decomposition of $S$.
Let us fix a smooth realization of $\a$ transverse to $\partial S$, and call $P_1,\ldots, P_k$ the complementary regions. We denote by $2n_i$ the number of sides of the cell $P_i$, and we fix a cyclic numbering of the sides, with the convention that sides with an odd number come from the boundary $\partial S$.\par

 We introduce the map 
$$\begin{array}{cccc}
P: & \teich(S,b) & \longrightarrow & \poly(2n_1)\times\ldots\times\poly(2n_k) \\
 & [m] & \longmapsto & (P_1(m),\ldots,P_k(m))
\end{array}$$ 
where $P_i(m)$ is marked right angled polygon in $(S,m)$ corresponding to $P_i$. To be more precise, the $P_i(m)$ is obtained by cutting the geodesic arcs orthogonal to $\partial S$ that belong to the isotopy classes of $\a$. In particular, the $P_i(m)$'s depend only on the isotopy class of $m$.

\begin{notation}
To simplify, we denote by $\poly(n_1,\ldots,n_k)$ the product $\poly(2n_1)\times\ldots\times\poly(2n_k)$.
We denote by $\ell_{i,j}:\poly(n_1,\ldots,n_k)\rightarrow\R$ the length function of the $j$-th side of the $i$-th polygon, so $\ell_{i,j}(P_1,\ldots,P_k)=\ell_j(P_i)$. As usual, we do not distinguish $\poly(n_1,\ldots,n_k)$ and its image in $\R^{2n_1}\times\ldots\times \R^{2n_k}$ \emph{via} the embedding given by the side length functions.
\end{notation}

 The map $P$ is clearly smooth and injective. Its image is caracterized by the following condition~: \emph{equality of the lengths of sides coming from the same arc}. The identifications between the sides of the $P_i$'s define an involution of $\{(i,j)~;~1\leq i\leq k\textnormal{ and } 1\leq j\leq 2n_i\}$, and a linear involution $\Psi$ of $\R^{2n_1+\ldots+2n_k}$. This involution is an isometry with respect to the canonical Euclidean metric. Therefore the fixed points locus is a totally geodesic submanifold, that is $P(\teich(\Sigma))\subset\poly(n_1,\ldots,n_k)$ is a totally geodesic submanifold.\par
 
\begin{proposition}
The map $P:\teich(\Sigma)\rightarrow \poly(n_1,\ldots,n_k)$ is a smooth embedding whose image is $\fix(\psi)$.
\end{proposition}

\begin{remark}
In particular, $P$ is proper.
\end{remark}

\begin{proof}
Let $\hat S$ be the closed oriented double cover of $S$, the arcs system $\a$ lifts to a curve system of $\hat S$. We can extend this curve system to a pants decomposition. The proposition comes immediately from the fact that Fenchel-Nielsen coordinates are coordinates (define a global chart of $\teich(\hat S)$).
\end{proof}

\subsection{The boundary length function}
Let $\B:\poly(n_1,\ldots,n_k)\rightarrow\R$ be the function defined by:
 \begin{eqnarray*}
 \B & = & \sum_{1\leq i\leq k} \sum_{j\mathrm{odd}} \ell_{i,j}.
 \end{eqnarray*}
It satisfies the identity
$$\B\circ P=\ell_\partial$$
over $\teich(S)$, where $\ell_\partial$ is the total length of the boundary. 
We observe that $\B$ is the restriction to $\poly(n_1,\ldots,n_k)$ of a linear form over $\R^{2n_1+\dots+2n_k}$. We deduce that its restriction to $P(\teich(\Sigma))$ is  a proper submersion, therefore it is a fibration according to Ehresmann theorem.

Using the equality of \textsection~\ref{sec:relation} we express the differential $\mrm d \B$ as a linear combination of the $\mrm d \ell_{i,j}$ with $j$ even~:
\begin{eqnarray*}
\mrm d \B & = & \sum_{1\leq i\leq k} \sum_{j\mathrm{odd}}\mrm d \ell_{i,j},\\
 & = & - \sum_{1\leq i\leq k} \left(\frac{1+\cosh\ell_{i,2}}{1+\cosh\ell_{i,1}}\cdot \frac{\sinh\ell_{i,1}}{\sinh\ell_{i,2}}\right) \left(\sum_{j\mathrm{even}}\mrm d \ell_{i,j}\right)
\end{eqnarray*}
over $\cup_{\V(n_1,\ldots,n_k)} T_P\poly(n_1,\ldots,n_k)$ where $\V(n_1,\ldots,n_k)=\V(2n_1)\times\ldots\times \V(2n_k)$.

\subsection{Proof of the theorem}

\begin{lemma}
For every $b>0$ there exists a unique point $[m]$ of $\teich(S,b)$ such that $P(m)$ belongs to $\V(n_1,\ldots,n_k)$.\end{lemma}

\begin{proof}
First, we observe that the points in $P(\teich(\Sigma))\cap \V(n_1,\ldots,n_k)$ are exactly points in $\V(n_1,\ldots,n_k)$ that have all their even sides (the one corresponding to arcs) of the same length (independently of $P_i$). We deduce that $P(\teich(\Sigma))\cap \V(n_1,\ldots,n_k)$ is a one dimensional submanifold of $\teich(S)$. It is parametrized by the length $\ell_{1,2}$ of even sides. On $P(\teich(\Sigma))\cap \V(n_1,\ldots,n_k)$ we have~:
$$\B=2\sum_{i=1}^k n_k \mrm{arcsinh}\left(\frac{\cos(\pi/n_k)}{\sinh(\ell_{1,2}/2)} \right).$$
So $\B$ is a strictly decreasing function of $\ell_{1,2}$ which tends to $+\infty$ (resp. $0$) when $\ell_{1,2}$ tends to $0$ (resp. $+\infty$). This proves the existence and unicity of $[m]$.\par
\end{proof}

\begin{lemma}
The components of $\a$ are exactly the arcs of minimal length of $[m]$.
\end{lemma}

\begin{proof}
We work with the polygon $P_i(m)$. The distance between two non adjacent sides of $P_i(m)$ is equal to the length of their common perpendicular. In order to compute this length, we first determine the length between these sides and the center $C$ of the polygon $P_i(m)$.\par
 Let us consider two consecutive sides of $P_i(m)$, the segments orthogonal to these sides and emanating from $C$ bound a trirectangle $T$ in $P_i(m)$. In this trirectangle, the sides opposite to $C$ have length $\ell_{i,1}/2$ and $\ell_{i,2}/2$, and the angle at $C$ is $\pi/n$. Let $h_{i,1}$ (resp. $h_{i,2}$) be the distance between $C$ and any side with odd number (resp. even number). Using classical trigonometric formulas we find:
\begin{eqnarray*}
\cosh(h_{i,1}) & = & \frac{\cosh(\ell_{i,2}/2)}{\sin(\pi/n)}.
\end{eqnarray*}
Now we consider the common perpendicular $\d$ between two non adjacent sides $a$ and $b$ of $P_i(m)$. The segment $\d$ together with the two segments emanating from $C$ orthogonal to $a$ and $b$ bound a pentagon in $P_i(m)$ with right angles except at $C$. Our aim is to compute the length $\ell(\d)$ of the side opposite to $C$. We distinguish two cases: either the sides $a$ and $b$ have both length $\ell_{i,1}$, either one has length $\ell_{i,1}$ and the other $\ell_{i,2}$.\par

 If they both have length $\ell_{i,1}$, the the pentagon has a symmetry whose axis is the segment orthogonal to $\d$ emanating from $C$. In this trirectangle we find:
\begin{eqnarray*}
\cosh(\ell(\d)/2) & = & \cosh h_{i,1} \sin(k\pi/n), \\
 & \geq &  \cosh h_{i,1} \sin(\pi/n)~=~\cosh(\ell_{i,2}/2).
\end{eqnarray*}
So $\ell(\d)\geq \ell_{i,2}$ with equality if and only if $\delta$ is a side (with even index) of $P_i(X)$.\par

 If one side is of length $\ell_{i,1}$ and the other of length $\ell_{i,2}$, then the formula iii) in the example~2.2.7 of \cite{buser} gives
\begin{eqnarray*}
\cosh(\ell(\d)/2) & = & -\cosh h_{i,1}\cosh h_{i,2}\cos(k\pi/n_i)+ \sinh h_{i,1}\sinh h_{i,2}.
\end{eqnarray*}
In the trirectangle $T$, we find the relation
\begin{eqnarray*}
\sinh(h_{i,j}) & = & \sinh(\ell_{i,j+1}/2)\cosh h_{i,j+1} ,
\end{eqnarray*}
and we finally obtain 
\begin{eqnarray*}
 \cosh(\ell(\d)/2) & = & \cosh h_{i,1}\cosh h_{i,2} \left( -\cos(k\pi/n_i)+ \sinh(\ell_{i,1}/2)\sinh(\ell_{i,2}/2) \right),\\
  & = &  \cosh h_{i,1}\cosh h_{i,2} \left( -\cos(k\pi/n_i)+ \cos(\pi/n_i) \right), \\
  & \geq & \cosh h_{i,1}\cosh h_{i,2} \left( -\cos(3\pi/n_i)+ \cos(\pi/n_i) \right), \\
  & \geq &  \cosh( \ell_{i,1}/2) \cosh( \ell_{i,2}/2) \frac{4 \cos(\pi/n_i)(1-\cos^2(\pi/n_i))}{\sin^2(\pi/n_i)},\\
  & \geq & \cosh(\ell_{i,1}/2)\cosh(\ell_{i,2}/2) 4 \cos(\pi/3), \\
  & \geq & 2\cosh(\ell_{i,1}/2) \cosh(\ell_{i,2}/2).
\end{eqnarray*}
So $\ell(\d)>\ell_{i,2}/2$. Let $c$ be a geodesic arc that do not retract on $\partial S$. Either this arc is isotopic to an element of $\a$, either there is a segment of $c$ that joined two non adjacent sides of one of the $P_i(m)$. The conclusion comes directly from the estimates above. 
\end{proof}

\begin{lemma}
 The point $[m]$ is eutactic, and its rank is equal to the number of components of $\a$.
\end{lemma}

\begin{proof}
By definition, the function $\B$ is constant on $P(\teich(S,b))$, thus $(\mrm d \B)_{[m]}\equiv 0$ over the tangent space $T_{[m]}(\teich(S,b))$. On the other hand, we have expressed in the previous paragraph $(\mrm d \B)_{[m]}$ as a linear combination with strictly negative coefficients of the $\mrm d \ell_{i,j}$ with $j$ even. These functions are exactly the length functions of the sides corresponding to the components of $\a$. This shows that $[m]$ is eutactic.
\end{proof}

\newpage
\part{Systole of geodesic loops}   

\section{Teichmüller spaces of pointed surfaces}\label{sec:teich-pointed}

Let $S$ be a closed oriented surface with negative Euler characteristic. Given a finite set of points $p=\{p_1,\ldots, p_n\}\subset S$, the \emph{Teichm\"uller space} of the pointed surface $(S,p)$ is the space of hyperbolic metrics on $S$ up to isotopy fixing $p$, we denote it by $\teich(S,p)$. The \emph{modular group} $\Mod(S,p)$ is the group of orientation preserving diffeomorphisms of the pair $(S,p)$ up to isotopy fixing $p$. By \emph{isotopy fixing} $p$, we mean an isotopy fixing each element of $p$ at all times.\par

\subsection{Spaces of representations}
We consider the universal cover $\tilde S$ of $S$ as the space of isotopy classes of paths with origin $p_1$. We denote by $\tilde p_1\in \tilde S$ the point corresponding to the class of the constant path. Note that $\pi_1(S,p_1)$ acts on the left of $\tilde S$ by precomposition. Given another pointed surface $(S',p_1')$, any continuous map $f:(S,p_1)\rightarrow (S',p_1')$ admits a unique lift $\tilde f:(\tilde S,\tilde p_1)\rightarrow (\tilde S',\tilde p_1')$, and defines a unique morphism $f_\ast:\pi_1(S,p_1)\rightarrow \pi_1(S',p_1')$. Clearly the lift $\tilde f$ is equivariant with respect to $f_\ast$.\par
 
 We denote by $\D_n$ the subset of $\Hyp^n$ that consists in $n$-tuples of points that are not all pairwise distinct. We fix a lift $\tilde p_i$ of each $p_i$ with $i\geq 2$, and a point $x\in \Hyp^n-\Delta_n$. We denote by $\rep(S)$ the space of discrete and faithful representations of $\pi_1(S,p_1)$ into $\Isom(\Hyp)$. Given a representation $\rho\in\rep(S)$, we denote by $x^\rho=\{x_1^\rho,\ldots,x_n^\rho\}$ the image of $x$ in $S^\rho=\Hyp/\rho$. The identity of $S^\rho$ admits a \emph{unique} lift $(\tilde S_\rho,\tilde x_1^\rho)\rightarrow (\Hyp,x_1)$, which is equivariant with respect to a \emph{unique} isomorphism $\pi_1(S_\rho,x_1^\rho)\rightarrow \rho(\pi_1(S))$. Let $f:(S,p)\rightarrow (S_\rho, x^\rho)$ be any diffeomorphism satisfying the following conditions:
\begin{itemize} 
\item $\rho:\pi_1(S)\rightarrow \rho(\pi_1(S))$ is the composition of $f_\ast:\pi_1(S,p_1)\rightarrow (S_\rho,x_1^\rho)$ with the isomorphism $\pi_1(S_\rho,x_1^\rho)\rightarrow \rho(\pi_1(S))$ above.,
\item $\tilde f(\tilde p_i)=x_i$ for each $i=1,\ldots, n$,
\end{itemize} 
we denote by $m_\rho$ the pullback by $f$ of the hyperbolic metric on $S_\rho$. The isotopy class of  $f$ is determined by $f_\ast$ and $\tilde f(\tilde p)$, so the isotopy class of $[m_\rho]_p$ depends only on $\rho$ and $x$. So we have a well-defined map $\f_x:\rep(S)\rightarrow \teich(S,p)$.\par
 If we allow the $x$ to vary, we obtain a map
$$\begin{array}{cccc}
\f: & \rep(S)\times (\Hyp^n-\Delta_n) & \longrightarrow & \teich(S,p) \\
 & (\rho,x) & \longmapsto & \f_x(\rho)
\end{array}$$ 
which invariant under the right action of $\Isom(\Hyp)$ on $\rep(S)\times (\Hyp^n-\Delta_n)$ defined by 
$$(\rho,x)\cdot \g~=~(\g^{-1}\rho\g,\g^{-1}(x))\quad (\forall \g\in\Isom(\Hyp)).$$
We denote by $\Phi:(\rep(S)\times (\Hyp^n-\Delta_n))/\Isom(\Hyp)\rightarrow\teich(S,p)$ the induced map.

\begin{proposition}
$\Phi$ is a homeomorphism. 
\end{proposition}

\begin{proof}
This map is clearly continuous and surjective. Let us show that it is injective, this is enough to conclude by invariance of the domain.\par
 Let $(\rho,x),(\rho',x')$ be two elements of $(\rep(S)\times (\Hyp-\Delta_n))/\Isom(\Hyp)$ with same images in $\teich(S,p)$.
We consider two diffeomorphisms $f:(S,p)\rightarrow (S_{\rho},x^{\rho})$ and $f':(S,p)\rightarrow (S_{\rho'},x'^{\rho'})$ satisfying the conditions above.\par 
We have that $\tilde f\circ(\tilde f')^{-1}$ conjugates the representation $\rho'$ and $\rho$ and sends $p'$ on $p$. By definition of $\teich(S,p)$, there exists an isotopy $t\mapsto F_t$ fixing $p$ such that $f'\circ F_1\circ f^{-1}:(S^\rho,x^\rho)\rightarrow (S^{\rho'},x'^{\rho'})$ is an isometry. 
The lift of this isometry conjugate $\rho$ and $\rho'$. Moreover, as $F$ fixes $p$, we have that its lift $t\mapsto \tilde F_t$ fixes each lift $\tilde p_i$. So finally we find that $\tilde f'\circ \tilde F_1 \circ \tilde f^{-1}(x)=\tilde f'\circ \tilde F_0 \circ \tilde f^{-1}(x)=x'$.
\end{proof}

We endow $\teich(S,p)$ with the smooth structure coming from the $\Isom(\Hyp)$-principal bundle.
 Note that any smooth section $s:\teich(S)\rightarrow \rep(S)$ gives a homeomorphism between $s(\teich(S))\times (\Hyp^n-\Delta_n)$ and $\teich(S,p)$. We can interpret this homeomorphism as follows

\subsection{Decomposition of $\teich(S,p)$ as a product}
 The space $\Conf(S)$ of conformal structures on $S$ is the total space of a $\Diff_0(S)$-bundle over the base $\teich(S)$, where $\Diff_0(S)$ is the group of diffeomorphisms of $S$ isotopic to the identity. As $\teich(S)$ is contractible, this bundle is trivial, and there exists a section $s:\teich(S)\rightarrow \Conf(S)$.\par

  We assume that $s$ is fixed. To any hyperbolic metric $m$ on $S$ corresponds a unique diffeomorphism $f_m\in\Diff_0(S)$ such that $m=(f_m)_\ast s([m])$. The existence is obvious as $m$ and $s([m])$ project on the same isotopy class $[m]\in\teich(S)$. The unicity comes from the fact that two homotopic isometries of a compact hyperbolic surface are equal, because their lifts coincide on the visual boundary of the universal cover.
Let $t\mapsto F_t$ be any isotopy from $id_S$ to $f_m$, we denote by $\push(m)$ the collection of homotopy classes of the paths $t\mapsto F_t(p_i)$ on $S$. We fix a lift $\tilde p_i\in \tilde S$ of each $p_i$, and we look at $\push(m)$ as a $n$-tuple of distinct points in $\tilde S$. The situation is summarized in the following diagram, where we denote by $[m]_p$ the class of $m$ in $\teich(S,p)$, and by $[m]$ its class in $\teich(S)$. 

\begin{tikzcd}
 & { s([m])\in \Conf(S) }   \arrow[]{r}{}[swap]{} & f_m\in\Diff_0(S)  \arrow[]{r}{}[swap]{}  &\push(m)\in(\tilde S^n-\Delta) \\
 m\in\Conf(S) \arrow[]{r}{}[swap]{}  &    {[m]_p}\in  \teich(S,p)   \arrow[]{u}{}[swap]{}  \arrow[]{d}{}[swap]{}   \arrow[]{rru}{}[swap]{}  &  \\
 &  {[m]\in \teich(S)} 
\end{tikzcd}

The following lemma shows that the application $\push:\Conf(S)\rightarrow (\tilde S^n-\Delta)$ is well-defined.

\begin{lemma}
The homotopy class $\push(m)$ does not depend on the choice of the isotopy $F$.
\end{lemma}

\begin{proof}
Let $t\mapsto G_t$ be another isotopy from $id_S$ to $f_m$, that gives a of path $(t\mapsto G_t(p_i)$ for each $1\leq i\leq n$.
The concatenation of $t\mapsto F_t$ and $t\mapsto G_{1-t}$ gives a loop in $\Diff_0(S)$, which is homotopically trivial for $\Diff_0(S)$ is  contractible. This implies that the concatenation of $t\mapsto F_t(p_i)$ and $t\mapsto G_{1-t}(p_i)$ is a loop which is trivial in $\pi_1(S,p_i)$. It follows that the paths $t\mapsto F_t(p_i)$ and $t\mapsto G_t(p_i)$ are homotopic.
\end{proof} 

The following lemma shows that $\push:\Conf(S)\rightarrow (\tilde S^n-\Delta_n)$ factorizes through $\teich(S,p)$.
 
\begin{lemma}
If two hyperbolic metrics  $m$ and $m'$ on $S$ satisfy $[m]_p=[m']_p$, then $\push(m)=\push(m')$.
\end{lemma} 
 
 \begin{proof}
 Let $F$ be an isotopy from $id_S$ to $f_m$.
 As $[m]_p=[m']_p$, there is an isotopy $G$ from $f_m$ to $f_{m'}$ that fixes $p$.
The concatenation of $G$ and $F$ gives an isotopy $F'$ between $id_S$ and $f_{m'}$. As $G$ fixes $p_i$, the paths defined by $F'$ and $F$ are the same up to the parametrization, so they define the same homotopy class.
 \end{proof}
 
 We now consider the map 
 $$\begin{array}{lrlc}
 \Psi: & \teich(S,p) & \longrightarrow &  \teich(S)\times (\tilde S^n-\Delta_n) \\
  & [m]_p & \longmapsto & ([m],\push(m))
\end{array}.$$
 
\begin{proposition}
The map $\Psi$ is a diffeomorphism between $\teich(S,p)$ and $\teich(S)\times (\tilde S^n-\Delta_n)$. In particular, $\push$ is injective on each isotopy class $[m]$.
\end{proposition}

\begin{proof}
One easily proves that $\Psi$ is a proper submersion. So it remains to show that $\Psi$ is injective.
Let $[m]_p$ and $[m']_p$ two points in $\teich(S,p)$ with same image $\Psi(m)=\Psi(m')$. From $[m]=[m']$, there exists an isotopy $f\mapsto F$ from $id_S$ to a diffeomorphism $F_1$ that realizes an isometry between $m$ and $m'$. As $\push(m)=\push(m')$, the homotopy class of each path $t\mapsto F_t(p_i)$ is trivial. So we modify $F$ in a contractible neighborhood of each $p_i$ so that it fixes $p_i$. We conclude that $[m]_p=[m']_p$.
\end{proof}

\subsection{Shearing coordinates}
The identification of $\teich(S,p)$ with $\teich(S)\times(\tilde S^n-\D_n)$ is constructed from a section $s:\teich(S)\rightarrow \Conf(S)$. We would like to make some explicit computations on $d_{\tilde m}(a,b)$ as $m$ varies in $s(\teich(S))$ and $a,b$ varies in $\tilde S$. For instance, given $\g\in\pi_1(S,p_1)$ we would like to compute the derivatives of $d_{\tilde m}(a,\g\cdot a)$, which is the length of the geodesic loop in the homotopy class $\g$ based at the projection of $a$. The aim of this paragraph is to explain how we can perform such computations using the shearing coordinates.\par

 Let us fix a hyperbolic metric $m$ and an isometry $\Theta:(\tilde S,\tilde m)\rightarrow \Hyp$. This isometry defines a representation $\rho:\pi_1(S)\rightarrow \Isom(\Hyp)$. We still denote by $\tilde\l$ the image $\Theta(\tilde \l)$.
Given a maximal geodesic lamination $\l$ on $S$, we have seen that there is an embedding $\Sigma:\teich(S)\rightarrow \Hold$, where $\Hold$ is the linear space of transverse H\"older distributions for $\l$. Let us recall few facts about it following the exposition of Bonahon and S\"ozen (\cite[\textsection 3]{bonahon-ergodic}).\par

Let $U$ be a sufficiently small neighborhood of the origin in $\Hold$. Fro any $\a\in U$ Bonahon has constructed a map $E^\a:\Hyp-\tilde \l\rightarrow \Hyp$ which is an isometry on each component of $\Hyp-\tilde\l$. This map conjugates the representation $\rho$ to another discrete and faithful representation $\rho_\a:\pi_1(S)\rightarrow \Isom(\Hyp)$. The quotient $\Hyp/\rho_\a$ is a hyperbolic surface diffeomorphic to $S$. Let $m_\a$ be the pull-back of the hyperbolic metric on $\Hyp/\rho_\a$ by any diffeomorphism $f:S\rightarrow \Hyp/\rho_\a$ whose induced homomorphism between fundamental groups is $f_\ast=\rho_\a$. The point $[m_\a]\in\teich(S)$ depends only on $\a$ and satisfies $\Sigma(m_\a)=\Sigma(m)+\a$.\par
 
 For any $\a\in U$ we set $\Theta^\a=E^\a\circ \Theta:(\tilde S-\tilde \l,\tilde m) \rightarrow \Hyp$, and we denote by $F^\a:(\tilde S-\tilde \l,\tilde m^\a)\rightarrow (\tilde S-\tilde \l,\tilde m)$ the unique isometry that extends to the identity on the visual boundary $\tilde S_\infty$.
 The map 
$$\begin{array}{cccc}
\G: \bigcup_{\a\in U} \{\a\}\times(\tilde S-\tilde \l,\tilde m_\a)\subset U\times \tilde S  & \longrightarrow & U\times \Hyp \\
(\a, x) & \longmapsto & (\a , \Theta^\a\circ F^\a(x))
 \end{array}$$
 is smooth and its restriction to each slice $ \{\a\}\times(\tilde S-\tilde \l,\tilde m_\a)$ is an isometry. Moreover this isometry is equivariant with respect to the actions of $\pi_1(S,p_1)$ on $\tilde S$ and on $\Hyp$ though $\rho_\a$.\par

\section{The hessian of the length function of a geodesic loop}

Let $S$ be a closed oriented surface with negative Euler characteristic, and $\l$ be a maximal geodesic lamination on $S$. We consider the following function:
$$\begin{array}{lcll}
\Dcal : & \teich(S)\times (\tilde S-\tilde \l)^2 & \longrightarrow & \R_+^\ast \\
           & ([m], p,q) & \longmapsto & d_{\tilde m}(p,q).
\end{array}$$
This function is smooth at any point $([m], p,q)$ with $p\neq q$. The aim of this section is to compute the hessian of $\Dcal$ with respect to the shearing coordinates on $\teich(S)$ and the hyperbolic metric on $\tilde S-\tilde \l$. Note that the function $\Dcal$

 We identify the tangent space $T_{[m]}\teich(S)$ with $\Hold$.\par
 Let us fix some notations. Given a leaf $l$ of $\tilde\l$ that intersects $[p,q]$, we orient $l$ so that it points to the left when one goes from $p$ to $q$, we denote by $\theta_l\in (0,\pi)$ the angle at $l\cap [p,q]$ between the vector that points in the direction of $q$ and $l$, we set $\ell_{pl}=d_{\tilde m}(p,l\cap [p,q])$ and $\ell_{lq}=d_{\tilde m}(l\cap [p,q],q)$. Given two leaves $l$ and $h$ of $\tilde \l$, we set $\ell_{p\{l,h\}}=\min(\ell_{pl},\ell_{ph})$ and $\ell_{\{l,h\}q}=\min(\ell_{lq},\ell_{hq})$.

\begin{theorem}\label{thm:hessian}
For any $\a\in \Hold$, any $u\in T_p\tilde S$, and any $v\in T_q\tilde S$, we have
\begin{eqnarray*}
\partial_{[m]}\Dcal_{([m],p,q)} (\a,) & = & \int_{[p,q]} \cos\theta_l\ \mrm d \a(l), \\
\partial_{(p,q)}\Dcal_{([m],p,q)} (u,v) & = & \|u\| \cos\psi_u + \|v\| \cos\psi_v,\\
\partial^2_{[m]} \Dcal_{([m],p,q)} (\a^2) & =& \frac{1}{\sinh \ell_{pq}}\ \int \int_{[p,q]^2} \cosh \ell_{p\{l,h\}} \cosh\ell_{\{l,h\}q} \sin\theta_l \sin\theta_h \mrm d \a(l) \mrm d\a(h) \\
\partial_{[m]}\partial_{(p,q)} \Dcal_{([m],p,q)} (\a,u,v) & = &  \frac{\|u\| \sin \psi_u}{\sinh\ell_{pq}} \int_{[p,q]} \cosh\ell_{lq} \sin\theta_l \mrm d \a(l) + \frac{\|v\| \sin \psi_v}{\sinh\ell_{pq}} \int_{[p,q]} \cosh\ell_{pl} \sin\theta_l \mrm d \a(l) \\
\partial^2_{(p,q)} \Dcal_{([m],p,q)} (u,v)^2 & = & \coth\ell_{pq}\ \left( \|u\|^2\ \sin^2\psi_u +\|v\|^2 \ \sin^2\psi_v \right)-2\frac{\|u\| \|v\|\sin\psi_u \sin\psi_v}{\sinh\ell_{pq}}, \\
 & = & \frac{1}{\sinh\ell_{pq}} \left(\|u\|\sin\psi_u-\|v\|\sin\psi_v \right)^2 + \tanh(\ell_{pq}/2) \left( \|u\|^2\sin^2\psi_u+\|v\|^2\sin^2\psi_v \right)
\end{eqnarray*}
where $\psi_u$ (resp. $\psi_v$) is the oriented angle at $p$ (resp. at $q$) between the vector pointing in the direction of $q$ (resp. $p$) and $u$ (resp. $v$).
\end{theorem}

The last formula above gives an explicit expression of the hessian of the distance function $d_\Hyp:\Hyp\times \Hyp\rightarrow \R$. As we were not able to find such a formula in the literature (see for instance \cite[Theorem~2.5.8]{Thurston}), we think that it deserves a specific statement:

\begin{corollary}
Let $p$ and $q$ be two distinct points in $\Hyp$, and let us denote by $U$ (resp. $V$) the tangent vector at $p$ (resp. $q$) obtained by rotating by an angle $\pi/2$ the vector pointing in the direction of $q$ (resp. of $p$). Then the Hessian of the distance function $d_\Hyp(\cdot,\cdot)$ at $(p,q)$ is given by $Q(\langle u, U\rangle, \langle v,V \rangle)$ for any $u\in T_p\Hyp$ and $v\in T_q\Hyp$, where $Q:\R^2\rightarrow\R$ is the positive definite quadratic form whose matrix in the canonical basis is
$$\frac{1}{\sinh d_\Hyp(p,q)} \begin{pmatrix} \cosh d_\Hyp(p,q) & -1 \\ -1 & \cosh d_\Hyp(p,q)  \end{pmatrix}.$$
In particular, the Hessian of $d_\Hyp$ is positive semidefinite and its isotropic cone is 
$$\{  (u,v)\in T_p\Hyp\times T_q\Hyp~;~u,v\ are\ tangent\ to\ the\ geodesic\  (p,q)  \}.$$
\end{corollary}

 As the Hessian of $d_\Hyp$ is positive-semidefinite, the Hessian of $\Dcal$ can not be positive-definite. Actually the only obstruction comes from $d_\Hyp$:

\begin{corollary}
If the projection of $[p,q]$ on $(S,m)$ intersects every leaf of $\l$, then the Hessian of $\Dcal$ at $([m],p,q)$ is positive-semidefinite, and its isotropic cone is 
$$0_{\Hold}\times \{  (u,v)\in T_p\Hyp\times T_q\Hyp~;~u,v\ are\ tangent\ to\ the\ geodesic\  (p,q)  \}.$$
\end{corollary}

To prove this corollary we follow the ideas of \cite[\textsection 6]{gendulphe-shearing}, and we provide an explicit lower bound on the Hessian $(\hess\Dcal)_{([m],p,q)}(\a,u,v)$.

\begin{proof}
 We fix a finite number of leaves $l_1,\ldots l_N$ of $\tilde \l$ that intersect $[p,q]$ enumerated from $p$ to $q$. We denote by $E$ the space of linear combinations of the Dirac measures $\d_{l_1},\ldots, \d_{l_N}$. This is a subspace of the space of transverse H\"older distributions on $[p,q]$ with support contained in $[p,q]\cap \tilde \l$.\par
 We look at $\Dcal$ as a smooth function over $E\times \Hyp^2$. To any element $\sum_i a_i \d_{l_i}$ of $E$ corresponds the deformation of $\Hyp$ obtained by shearing of an amount $a_i$ along $l_i$ for each $i=1,\ldots, n$. The formulas of the theorem~\ref{thm:hessian} work also in this context, because they were first established for a shear along one leaf.\par
 The Hessian of $\Dcal$ (as a function over $E\times \Hyp^2$) can be written as a quadratic form in the variables $\sin\theta_1 a_1,\ldots,\sin\theta_n a_n,\langle u,U \rangle, \langle v,V \rangle $ whose matrix $H$ is given by 
$$H=
 \begin{pmatrix}
   \cosh\ell_{pl_1}\cosh\ell_{l_1 q}  & \ldots &  \cosh\ell_{pl_1}\cosh\ell_{l_n q} & &\cosh\ell_{pl_1} & \cosh\ell_{l_1 q} \\
   \vdots                                         &   \ddots          & \vdots  & &\vdots & \vdots \\
   \cosh\ell_{pl_1}\cosh\ell_{l_n q} &  \ldots & \cosh\ell_{pl_n}\cosh\ell_{l_n q} & & \cosh\ell_{pl_n} & \cosh\ell_{l_n q} \\ \\
  \cosh\ell_{pl_1} & \ldots &  \cosh\ell_{pl_n}  & & \cosh\ell_{pq} & -1 \\
  \cosh\ell_{l_1q} & \ldots &  \cosh\ell_{l_nq}  & & -1 & \cosh\ell_{pq}  \\   
 \end{pmatrix}.$$ 
The diagonal is strictly dominant ($a_{ii}>|a_{ij}|$ for any $i,j$), therefore $H$ is positive-definite (see \cite[Lemma 6.1]{gendulphe-shearing}). This implies immediately that the Hessian of $\Dcal$ over $E\times \Hyp^2$ is positive-semidefinite, and that its isotropic cone is $$0_E\times \{  (u,v)\in T_p\Hyp\times T_q\Hyp~;~u,v\ \textnormal{are\ tangent\ to\ the\ geodesic}\  (p,q)  \}.$$
Actually we can be more precise. For each leaf $l_i$ we denote by $\e_i$ the minimal distance on $[p,q]$ between $l_i$ and any other leaf of $\tilde \l$. We note $\e_p$ (resp. $\e_q$) the minimal distance on $[p,q]$ between $p$ (resp. $q$) and any leaf of $\tilde \l$. Finally, we denote by $H'$ the matrix which has same entries as $H$ outside the diagonal, and diagonal entries given by
$$\begin{array}{ccll}
 (H')_{i,i} & = & \cosh\ell_{pl_i}\cosh(\ell_{l_iq}-\e_i) & \textnormal{for } i=1,\ldots,n; \\
 (H')_{n+1,n+1} & = &\cosh(\ell_{pq}-\e_p) & \\
 (H')_{n+2,n+2} & = &\cosh(\ell_{pq}-\e_q) & \\
 \end{array}$$
By construction, $H'$ has a dominant diagonal ($|a_{ii}|\geq a_{ij}$ for any $i,j$), and consequently is positive-semidefinite.
The difference $D=H-H'$ is a diagonal matrix whose entries can be bounded from below as follows: 
\begin{eqnarray*}
\cosh\ell_{pl_i} (\cosh(\ell_{l_i q})-\cosh(\ell_{l_i q}-\e_i)) & \geq & \cosh\ell_{pl_i} \sinh(\ell_{l_i q}-\e_i) \e_i ,\\
\cosh\ell_{pq} -\cosh(\ell_{l_i q}-\e_p) & \geq & \sinh(\ell_{pq}-\e_p) \e_p.
\end{eqnarray*}
Note that the lower bounds depend only on $p,q$ and $\tilde \l$. This gives a lower bound for the Hessian of $(\hess\Dcal)_{([m],p,q)}$ evaluated at $(\sum a_i \d_{l_i},u,v)$.\par
 Now, let us fix $\a\in\Hold-\{0\}$, and vectors $u\in T_p\Hyp,v\in T_q\Hyp$. The transverse H\"older distribution $\a$ is the limit of a sequence $(\a_n)_n$ of linear combination of Dirac measures (see \cite[Lemma~3.1]{gendulphe-shearing}), associated to  an increasing sequence $(\mc P_n)_n$ of subsets of $\mc P_{pq}$. Let $l_1,\ldots,l_k$ be a sequence of isolated leaves of $\tilde \l$ that intersects $[p,q]$. For any $n$ big enough, $\mc P_n$ contains the components of $\tilde S-\tilde \l$ adjacent to the $l_i$'s, and the hessian of $(\hess\Dcal)_{([m],p,q)}$ evaluated at $(\a_n,u,v)$ is bounded from below by 
$$\left(\sum_{i=1,\ldots,k} \cosh\ell_{pl_i}\sinh(\ell_{l_iq}-\e_i)\e_i \sin^2\theta_{l_i}\right)+\sinh(\ell_{pq}-\e_p)\e_p \langle u,U\rangle^2+\sinh(\ell_{pq}-\e_q)\e_q \langle v,V\rangle^2 .$$
 This bound does not depend on $n$, and thus implies the corollary as $\sin\theta_l$ is positive for any leaf $l$.
\end{proof}

We will see applications of these result to the study of topological singularities of $\sys_p$ (\textsection ), and to the study of variations of various functions (\textsection ).

\section{Proofs of the formulas}

To prove these formulas we follow the same line of arguments as in \cite{gendulphe-shearing}: we compute explicitly the derivatives for a shear along one geodesic

\subsection{Shearing coordinates}
We have explained in how to make effective computations in the shearing coordinates. The method is due to Bonahon, and proceed by an approximation
We identity $(\tilde S,\tilde m)$ with $\Hyp$ and we use the notations of , that are essentially the notations introduced by Bonahon.
In particular, we consider the point $p$ fixed, whereas $q$ is moved by the isometry $\f$.

Notation: we perform a shearing along $l$, or $h$ but in both cases these geodesics are pointwise fixed. We denote by $g(t)$ the geodesic passing through $p$ and $q(t)$, given a leaf $l$ we denote by $l(t)$ the intersection point, we denote by $f'_l$ the function such that the velocity vector of $l\cap g(t)$ is $f'_l l'$ where $l'$ is the unit tangent vector field along $l$.

 We fix two points $p,q$ that belong to two distinct components $P,Q$ of $\tilde S - \tilde \l$. We denote by $\mc P_{P,Q}$ the set of components of $\tilde S -\tilde \l$ that separate $P$ from $Q$.

\subsection{First derivative}

Given a leaf $l$ that separate $p$ from $q$, we denote by $V(l)$ the vector field on $\tilde S$ defined by:
 \begin{eqnarray*}
V_z(l)& = & \left[\frac{\mrm d}{\mrm d t} T_l^t(z)\right]_{t=0},
 \end{eqnarray*}
 where $T_l^t$ is the isometry of $(\tilde S,\tilde m)$ that translates by a length $t$ along $l$ in the positive direction. We denote by $Z$ the unit vector field on $\tilde S-p$ pointing in the direction of $p$, note that the gradient of $z\mapsto d_{\tilde m}(p,z)$ is equal to $-Z$. The approximation method used in gives
 \begin{eqnarray*}
\partial\Dcal_{([m],p,q)}(\a) =- \int_{\g} \langle Z_q, V_q(l) \rangle\  \mrm d \a(l).
 \end{eqnarray*}
By construction $V(l)$ is Killing, which means that the associated one parameter subgroup of diffeomorphisms is a groups of isometries. We recall the following elementary property of Killing vector fields:

\begin{lemmanonumber}
Let $V$ be a Killing field vector field, and $\g$ be a geodesic. The scalar product $\langle V, \g' \rangle$ is constant.
\end{lemmanonumber}

As a consequence,  we have $\langle Z_q,V_q(l)  \rangle=\langle Z_{q'},V_{q'}(h) \rangle$ where $q'$ is the intersection point of $l$ with $[p,q]$.
As the restriction of $V(l)$ to $l$ is the unit tangent vector field, we find $-\langle Z_{q'},V_{q'}(h) \rangle=\cos\theta_l$ and we obtain 
\begin{eqnarray*}
\partial \Dcal_{([m],p,q)}(\a) = \int_{\g} \cos\theta_l \ \mrm d \a(l).
\end{eqnarray*}
We obviously have $\partial \Dcal_{([m],p,q)}(u,v)=\|u\| \cos\psi_u+\|v\|\cos\psi_v$.

\subsection{Second derivative in $\a$}\label{sec:derivative-alpha}

Let $h,l$ be two distinct leaves of $\tilde \l$ that separate $p$ from $q$. We want to compute the derivative of $\cos\theta_l$ when shearing along $h$. For symmetry reasons, we assume that $h$ separate $l$ from $q$. This is convenient because the geodesic $l$ is not moved by the shear along $h$.\par

 We denote by $q(t)$ the point $T^t_h(q)$, and by $g(t)$ the half-geodesic starting at $p$ and passing through $q(t)$. The geodesic $g(t)$ is obtained by rotating $g$ by an angle $\rho(t)$ at $p$. We will use the following obvious lemma:

\begin{lemma}\label{lem:angle}
Let $C$ be a smooth curve that intersects $g$ transversely in one point. We denote by $c(t)$ the intersection point $C\cap g(t)$ for $t$ sufficiently small. We have 
\begin{eqnarray*}
\|c'(0)\| & = & \frac{\rho'(0) \sinh d(p,c(0)) }{\sin\theta},
\end{eqnarray*}
where $\theta$ is the angle at $c(0)$ from $g$ to $c$. 
\end{lemma}

\begin{lemma}\label{lem:angle-shearing}
We have
\begin{eqnarray*}
\rho'(0) & = & \frac{\cosh\ell_{hq}\sin\theta_h }{\sinh\ell_{pq}}.
\end{eqnarray*}
\end{lemma} 
 
\begin{proof}
An elementary computation (with Fermi coordinates for instance) shows that 
$$\| q'(t)\|=\cosh r,$$
where $r$ is the distance from $q$ to $h$.
Let us denote by $s$ the orthogonal projection of $q$ on the geodesic $h$, so that $r=d_{\tilde m}(s,q)$.
In the right-angled triangle with vertices $q$, $s$ and $h\cap g$, we call $\psi$ the angle at $q$. The segment $[s,q]$ is orthogonal to the trajectory $t\mapsto q(t)$, thus the angle between $g$ and the trajectory $t\mapsto q(t)$ is equal to $\pi/2-\psi$. The above lemma gives
 $$\cosh r=\|q'(0)\| = \frac{\rho'(0) \sinh\ell_{pq}}{\cos\psi}. $$
In the right-angled triangle we have the relation $\cos\psi=\tanh r /\tanh \ell_{hq}$, that implies 
$$\rho'(0)= \frac{\sinh r}{\tanh \ell_{hq} \sinh\ell_{pq}}.$$
We conclude with the relation $\sinh r=\sinh \ell_{hq} \sin\theta_h$.
\end{proof} 
 
 Let us fix an origin on $l$ and denote by $f_l(t)$ the signed distance between this origin and $l \cap g(t)$. From the two previous lemmas we obtain
 $$f'_l=\frac{\cosh \ell_{hq}\sinh \ell_{pl} \sin\theta_h}{\sinh\ell_{pq}\sin\theta_l}.$$
Then the last equality in gives
\begin{eqnarray*}
\left[\frac{\mrm d }{\mrm dt} \cos\theta_l \right]_{t=0} & = & -f'_l \sin^2\theta_l \coth(\ell_{pl}), \\
 &= & \frac{\cosh\ell_{pl}\cosh\ell_{hq}}{\sinh\ell_{pq}} \sin\theta_l  \sin\theta_h.
\end{eqnarray*}

\subsection{Second derivative in $(u,v)$}
We have to differentiate the function $(p,q)\mapsto \cos \psi_u$.
The last equality in \cite{gendulphe} gives immediately 
\begin{eqnarray*}
\partial_p (\cos\psi_u)(u) &=& \|u\| \sin^2\psi_u \coth\ell_{pq} .
\end{eqnarray*}
Using the Lemma~\ref{lem:angle} above we find
\begin{eqnarray*}
\partial_q(\psi_u)(v) & = &  \frac{ \|v\|\sin\psi_v}{\sinh\ell_{pq}},
\end{eqnarray*}
so that
\begin{eqnarray*}
\partial_q (\cos\psi_u)(v)  & = & -\sin\psi_u \ \partial_q(\psi_u)(v),\\
 & = & -\frac{\|v\| \sin\psi_u \sin\psi_v}{\sinh\ell_{pq}}. 
\end{eqnarray*}

\subsection{Second derivative in $\a$ and $(u,v)$}
We have to differentiate $\cos\psi_u$ as $q$ moves according to the deformation defined by $\a$. The Lemma~\ref{lem:angle-shearing} gives 
\begin{eqnarray*}
\partial_{[m]}( \psi_u)(\a) & = &- \frac{1}{\sinh\ell_{pq}}  \int_{[p,q]} \cosh\ell_{lq}\sin\theta_l\ \mrm d \a(l).
\end{eqnarray*}
So we find
\begin{eqnarray*}
\partial_{[m]}(\cos\psi_u)(\a) & = &- \sin\psi_u\  \partial_{[m]}( \psi_u)(\a),\\
 &= & \frac{\sin\psi_u}{\sinh\ell_{pq}}  \int_{[p,q]} \cosh\ell_{lq}\sin\theta_l\ \mrm d \a(l).
\end{eqnarray*}
Similalry we have
\begin{eqnarray*}
\partial_{[m]}(\cos\psi_v)(\a) & = & \frac{\sin\psi_v}{\sinh\ell_{pq}}  \int_{[p,q]} \cosh\ell_{pl}\sin\theta_l\ \mrm d \a(l).
\end{eqnarray*}

\section{Extreme points of $\sys_p$}

\begin{theorem}
Let $S$ be a compact surface with $k\geq 0$ boundary components $b_1,\ldots, b_k$. A point $[m]\in\teich(S,p)$ is extreme for $\sys_p$ if and only if its systolic loops at $p$ divide $S$ into equilateral triangles and one holed monogons. In that case $\sys_p(m)$ is the unique positive solution of the equation
$$6(-2\chi(S)+2-k) \mathrm{arcsin}\left(\frac{1}{2\cosh(x/2)} \right)+ 2\sum_{i=1}^k \mrm{arcsin}\left(\frac{\cosh(\ell(b_i)/2)}{\cosh(x/2)} \right)~ =~2\pi. $$
\end{theorem}

\begin{proof}
The existence of a global maximum comes from standard arguments. Let $[m]_p$ be a metric with $\sys(m)$ very small, then no $m$-systolic loop at $p$ intersect the geodesics realizing the systole. So we can cut this geodesic, we increase the length of the corresponding boundary components, and we glue them back. This gives another hyperbolic surfaces $(S',m')$ for which there is a point $p'$ with $\sys_{p'}(m')>\sys_p(m)$. We refer to \cite{gendulphe-schwartz} for more details.\par
Let $[m]$ be an extreme point, then $[m]$ has at least $\dim\teich(S,\partial S)+1=-3\chi(S)+3-k$ systolic loops at $p$ (Theorem~\ref{thm:voronoi}). By minimality of their length, these systolic loops do not intersect outside $p$. 
But, an easy computation of Euler characteristic shows that the minimal cardinality of a set of non homotopic loops that intersect only at $p$ is also $-3\chi(S)+3-k$. So we deduce that the $m$-systolic loop at $p$ decompose $S$ into triangles and monogons. 
The number of triangles is $-2\chi(S)+2-k$, and the number of monogons is $k$. This shows one implication, the other implication is trivial.
\end{proof}

\section{Topological singularities of $\sys_p$}

 Any homotopy class $\g\in \pi_1(S,p)$ defines a smooth length function
$$\ell_\g:\teich(S,p)\rightarrow \R.$$

Let $\l$ be a maximal lamination of $S$ such that intersects transversely $\g$.

\begin{lemma}
Let $m$ be a hyperbolic metric on $S$ such that $p$ does not belong to the $m$-realization of $\l$. Then Hessian $(\hess\ell_\g)_{[m]_p}$ in the shearing coordinates with respect to $\l$ is positive-semidefinite, and positive-definite if and only if $\sys_p(m)$ is not realized by a closed $m$-geodesic passing through $p$.  
\end{lemma}

\begin{proof}
We have $\ell_\g(m)= d_{\tilde S}(\tilde p,\g\tilde p)$, which gives the equality
$$(\hess \ell_\g)_{[m]_p}(\a,u)= (\hess\Dcal)_{([m],\tilde p,\g\tilde p )}(\a,u,\g\cdot u) $$
So $(\hess \ell_\g)_{[m]_p}(\a,u)=0$ if and only if $\a=0$ and $u,\g\cdot u$ are tangent to the geodesic $(\tilde p, \g(\tilde p))$. Let $f$ be a non elliptic isometry of $\Hyp$, and $x$ be apoint in $\Hyp$. Given $u\in T_p\Hyp$ the isometry $f$ sends the geodesic tangent to $u$ at $x$ to the geodesic tangent to $(\mrm df)_x(u)$ at $f(x)$. So if $u$ and $f\cdot u$ are tangent to $(x,\g(x))$ then this geodesic is glovally preserved by $f$, which implies that $f$ is hyperbolic and $u$ is tangent to its axis.
\end{proof}

\begin{definition}
We say that a point $[m]_p\in\teich(S,p)$ is a \emph{singularity associated to a short geodesic} if there is a closed $m$-geodesic passing through $p$ that realizes $\sys_p(m)$.
\end{definition}

In particular, if a systolic loop of $m$ is passing through $p$, then $[m]_p$ is a singularity associated to a short geodesic. More generally the set of such singularity in a fiber $[m]$ is a $1$-dimensional closed subset.

\begin{itemize}
\item if $p$ belongs to a geodesic realizing $\sys(m)$ then it is a short geodesic singularity,\item a point which realizes a local maximum of $\sys_p$ on a fiber $[m]$ is not a short geodesic siingularity,  in particular any extreme point of $\sys_p$ is not a short geodesic singularity.
\end{itemize}
 
 As a consequence we get:
 
\begin{theorem}
The restriction of $\sys_p:\teich(S,p)\rightarrow\R_+^\ast$ to $\teich(S,p)-$ is a topological Morse function whose critical points of rank $k$ are the eutactic points of rank $k$.   
\end{theorem}

In particular, we find that the extreme points of $\sys_p$ are characterized

\begin{corollary}
A point $[m]_p\in\teich(S,p)$ is extreme if and only if the systolic loops at $p$ divide $(S,m)$ into isometric equilateral triangles whose sides have length $\sys_p(m)$.
\end{corollary}

\begin{proof}
As we have seen, an extreme point $[m]_p$ has at least $\dim\teich(S,p)+1$ systolic loops at $p$. An easy computation shows that this is the maximal cardinality of a set of geodesic loops based at $p$, and that such a set divide the surface into equilateral triangles and monogons whose sides have length $\sys_p(m)$
\end{proof}

\section{Critical points of $\sys_p$}

\subsection{Three kinds of polygons}
We fix a hyperbolic metric $m$ on $S$. The lifts of the systolic loops at $p$ form a family of geodesics that divide $\tilde S$ into \emph{convex} $\tilde m$-polygons. These polygons fall into a finite number of $\pi_1(S,p)$-orbits, and they are of three kinds: 
\begin{itemize} 
\item \emph{compact polygons}, the interior of such a polygon projects isometrically onto a component of $(S,m)-\Sys(m)$ that do not contain a boundary component nor a cusp.
\item \emph{polygons that contain a boundary component of $\partial \tilde S$}, they have an infinite number of sides, but are invariant under the action of the cyclic subgroup of $\pi_1(S,p)$ preserving the boundary component. Such a polygon projects onto a component of $(S,m)-\Sys(m)$ that contains a boundary component.
\item \emph{polygons that contain an horoball}, they have an infinite number of sides, but are invariant under the action of the cyclic subgroup of $\pi_1(S,p)$ preserving the horoball. Such a polygon projects onto a component of $(S,m)-\Sys(m)$ that contains a cusp.
\end{itemize}

\subsection{Differentials of lengths and angles}
Let $x=(x_1,\ldots ,x_l)\in\Hyp^l-\D_l$ be a polygon. For each $i$, we denote by $\ell_i$ the length of the segment $[x_{i}, x_{i+1}]$, by $U_i$ (resp. $V_i$) the unit tangent vector at $x_i$ pointing in the direction opposite to $x_{i-1}$ (resp. $x_i$), and by
$\theta_i$ the oriented angle $\widehat{x_{i-1}x_i x_{i+1}}$, which is also the oriented angle from $U_i$ to $V_i$.\par
 From the expression of the differential of $d_\Hyp$, we find for any $u_i\in T_{x_i}$:
\begin{eqnarray*}
(\mrm d\ell_{i-1})_x(u_i) & = &  \langle u_i,U_i \rangle,\\
(\mrm d\ell_{i})_x(u_i) & = &  \langle u_i,V_i \rangle.
\end{eqnarray*}
In particular for $u=(u_1,\ldots, u_l)\in T_x\Hyp$ we have
\begin{eqnarray*}
(\mrm d\ell_{i})_x(u) & = &  \langle u_i,V_i \rangle+ \langle u_{i+1},U_{i+1}\rangle, \\
\sum_{i=1,\ldots,l} (\mrm d\ell_i)_x(u) & = & \sum_{i=1,\ldots,l}  \langle u_i,U_i+V_i \rangle,
\end{eqnarray*}
and we easily observe that:

\begin{lemma}
If each $\theta_i$ is different from $\pi$ and $0$, then the differentials $(\mrm d\ell_1)_x,\ldots,(\mrm d\ell_l)_x  $ are linearly independent in $T^\ast_x\Hyp^l$. In particular, the set of length-regular polygons with angles different from $0$ and $\pi$ is a submanifold $V(l)$ of dimension $l+1$ of $\Hyp^l$.
\end{lemma}

From Lemma~\ref{lem:angle} we deduce 
\begin{eqnarray*}
(\mrm d\theta_{i-1})_x(u_i) & = & \frac{\langle u_i, U_i^\perp \rangle}{\sinh\ell_{i-1}}, \\
(\mrm d\theta_{i+1})_x(u_i) & = & \frac{\langle u_i, V_i^\perp \rangle}{\sinh\ell_{i}}, 
\end{eqnarray*}
where $U_i^\perp$ (resp. $V_i^\perp$) is obtained by rotating $U_i$ (resp. $V_i$) by an angle $\pi/2$ (resp. $-\pi/2$). The formula at the end of \textsection\ref{sec:derivative-alpha} gives
\begin{eqnarray*}
(\mrm d \theta_i)_x(u_i) & = & \coth\ell_{i-1} \langle u_i, U_i^\perp \rangle + \coth\ell_{i}  \langle u_i, V_i^\perp \rangle,
\end{eqnarray*}
so that
\begin{eqnarray*}
(\mrm d \theta_i)_x(u) & = & \frac{\langle u_{i-1}, V_{i-1}^\perp \rangle}{\sinh\ell_{i-1}}+ \coth\ell_{i-1} \langle u_i, U_i^\perp \rangle + \coth\ell_{i}  \langle u_i, V_i^\perp \rangle + \frac{\langle u_{i+1}, U_{i+1}^\perp \rangle}{\sinh\ell_{i}}.
\end{eqnarray*}
In particular
\begin{eqnarray*}
\sum_{i=1,\ldots,l} (\mrm d\theta_i)_x(u) & = & \sum_{i=1,\ldots,l} \frac{1+\cosh\ell_{i-1}}{\sinh\ell_{i-1}}   \langle u_i,U_i^\perp \rangle + \frac{1+\cosh\ell_{i}}{\sinh\ell_{i}} \langle u_i,V_i^\perp \rangle  \\
& = & \sum_{i=1,\ldots,l} \coth\frac{\ell_{i-1}}{2}   \langle u_i,U_i^\perp \rangle +  \coth\frac{\ell_{i}}{2}  \langle u_i,V_i^\perp \rangle.
\end{eqnarray*}
If $x$ is length-regular, then we find that 
\begin{eqnarray*}
\sum_{i=1,\ldots,l} (\mrm d\theta_i)_x(u) & = & \coth\frac{\ell_{1}}{2} \sum_{i=1,\ldots,l}   \langle u_i,U_i^\perp+V_i^\perp \rangle \\
& = & -\coth\frac{\ell_{1}}{2}  \sum_{i=1,\ldots,l} \tan(\theta_i/2) (\mrm d\ell_i)_x(u)
\end{eqnarray*}
because $U_i^\perp+V_i^\perp=-\tan(\theta_i/2)(U_i+V_i)$.

\subsection{Spaces of polygons}

Given $l\geq 3$, the space of marked polygons of $\Hyp$ with $l$ vertices is $\Hyp^l-\D_l$. We denote by $\poly(l)$ the space of isometry classes of marked polygons of $\Hyp$ with $l$ vertices, which is the quotient of $\Hyp^l-\D_l$ by the diagonal action of $\Isom^+(\Hyp)$. This is a manifold of dimension $l-3$.\par
 Similarly there is a space $\poly(l,b)$ of isometry classes of marked polygons with $l$ vertices and one hold with boundary length $b$. This is the of points on $\{z\in\Hyp;x<0\}^l/\Z$

 Let $\{\g_1,\ldots ,\g_k\}\subset \pi_1(S,p)$ be homotopy classes of simple loops intersecting only at $p$ that divide $S$ in polygons $P_1,\ldots,P_m$, eventually with on hole or cusp. This defines a smooth map from $\Phi:\teich(S,p)\rightarrow \poly_1\times\ldots\times \poly_m$, where $\poly_i$ is the space of isometry classes of marked polygons corresponding to $P_i$.
 
The image of $\Phi$ is the submanifold of codimension $1+k$ defined by the conditions: the total sum of the angles is equal to $2\pi$, two identified sides have same length.
As the map has an obvious smooth inverse, we conclude that
 
\begin{proposition}
The map $\Phi$ is a diffeomorphism onto its image
\end{proposition} 

Let us denote by $F$ the number of components of $S-\cup_i \g_i$. We have $\chi(S)=1-k+F$, so $2-2g-(b+s)=1-k+F$.
Note the length and angle functions are invariant under $\Isom(\Hyp)$, so they are well 
Note that the differentials of the length functions of sides are linearly independent at any point of $\poly_1\times \poly_m$, so the length the differential of the length functions of identified sides are linearly independent, this implies that the rank on the image of $\Phi$ is at least $k-1$, more precisely each family of $k-1$ is free.

\subsection{Proof of the theorem}

\begin{lemma}
If the systolic $m$-geodesic loops at $p$ divide $(S,m)$ in regular polygons, then $[m]_p$ is a critical point of $\sys_p$ whose rank is equal to the cardinal of $\Sys_p(m)$ minus $1$.
\end{lemma}

\begin{proof}
If the systolic $m$-geodesic loops at $p$ divide $(S,m)$ in regular polygons. We denote $P_1,\ldots ,P_m$ the polygons, note that the isometry class of each marked polygon is well-defined for each point of $\teich(S,p)$. For each polygon the differential of the total sum of the angles is equal to a negative scalar times the differential of the sum of the length. But the total angle at $p$ is constant equal to $2\pi$, this shows that $[m]_p$ is eutactic. The assertion on the rank is obvious from the discussion of the previous paragraph.\par
\end{proof}

\begin{lemma}
Let $\Sys\subset \pi_1(S,p)$ be a family of homotopy classes of simple loops that divide $S$ into polygons. Then the minimal class of $S$ contains exactly one eutactic point.
\end{lemma}

\begin{proof}
We have just proved the existence, so let us show the unicity. Let $[m]_p$ be an eutactic point of $\sys_p$. The $m$-systolic loops at $p$ divide $(S,m)$ into length-regular polygons $P_1,\ldots, P_k$ with angles less than $\pi$. We assume that $[m]_p$ is not one of the critical points of the previous lemma. We will show that $[m]_p$ does not realize strict local minimum of $\sys_p$ in its minimal class. This will contradict the Proposition~1.6 of \cite{bavard-smf}, and establish the lemma.\par

 As $[m]_p$ is not one of the critical points above, there is at least one of the $P_i$'s, let say $P_1$, which is not regular. Then we can deform it into a length-regular polygon $P'_1$ with same area (\cite[Lemma~4.2]{schlenker}). Note that we can choose $P'_1$ as close as we want from $P_1$.\par
 
  By gluing the polygons $P'_1,P_2,\ldots,P_k$ we determine a new point $[m']_p\in\teich(S,p)$. The points $[m]_p$ and $[m']_p$ belong to the same minimal class and satisfy $\sys(m')< \sys_p(m)$. We can construct $m'$ as close as we want from $[m]_p$, this show that $[m]_p$ is not a strict minimum of $\sys_p$ in its minimal class.
\end{proof}

\subsection{Degenerate points}

\begin{lemma}
Any point in $\short(S,p)$ is degenerate. 
\end{lemma}

\begin{proof}
As $\sys_p$ is a Morse function on $\teich(S,p)-\short(S,p)$, it comes that the points outside $\short(S)$ are regular or nondegenerate.
First we remark, from the definition of regular and critical nondegenerate points, that the set of degenerate points is closed. As $\short(S)$ is closed, we have just to show that it contains a dense subset of critical degenerate points.\par

 So we assume that $m$ is a metric such that $\sys_p(m)$ is realized by a closed geodesic $\g$ passing through $p$. Moreover we assume that this is the unique systolic geodesic loop at $p$. Let us show that this  last condition defines a dense subset of $\short(S)$. By cutting $\g$, we obtain two more boundaries with a marked point on each of them. Each systolic loop at $p$ gives an arc that joins the two boundary components. Let $\a$ be the shortest arc that connect the two boundaries (we assume $\g$ nonseparating, the separating case is easier), applying the peeling method (described in \cite{papadopoulos}) to the arc $\a$ we can shorten $\g$ more than any of the minimal arc. Then we reglue the boundary components, and we consider the point which is the orthogonal projection of $p$. \par
 
 So now the geodesic $\g$ is the unique systolic loop at $p$. Considering a system of Fenchel-Nielsen coordinates containing $\g$, we see that the set of point such that $\sys_p\leq \sys_p(m)$ in a neighborhood of $[m]_p$ is of dimension $\dim(\teich(S))/2$. So $[m]$ can not be a regular point, and it is degenerate because it is not an isolated critical point (consider the other points on $\g$ sufficiently close to $p$).
\end{proof}

\bibliographystyle{alpha}
\bibliography{biblio}

\end{document}